\documentclass[a4paper,11pt,leqno]{article}
\usepackage{amssymb, amsmath, amsthm, graphicx}

\newtheorem{thm}{Theorem}[section]
\newtheorem{lem}[thm]{Lemma}
\newtheorem{cor}[thm]{Corollary}
\newtheorem{prop}[thm]{Proposition}

\newtheorem{rem}{Remark}[section]

\numberwithin{equation}{section}

\renewcommand{\a}{\alpha}
\renewcommand{\b}{\beta}
\newcommand{\e}{\varepsilon}
\newcommand{\de}{\delta}
\newcommand{\fa}{\varphi}
\newcommand{\ga}{\gamma}
\renewcommand{\k}{\kappa}
\newcommand{\la}{\lambda}
\renewcommand{\th}{\theta}

\newcommand{\Si}{\Sigma}

\newcommand{\De}{\Delta}

\newcommand{\lan}{\langle}
\newcommand{\ran}{\rangle}

\def\R{{\mathbb{R}}}
\def\N{{\mathbb{N}}}
\def\Z{{\mathbb{Z}}}

\allowdisplaybreaks

\title{Hydrodynamic limit for an evolutional model of \\
two-dimensional Young diagrams}
\author{Tadahisa Funaki and Makiko Sasada}
\date{September 30, 2009}
\begin{document}
\maketitle

\begin{abstract} 
\noindent
We construct dynamics of two-dimensional Young diagrams, which are naturally
associated with their grandcanonical ensembles, by allowing the creation
and annihilation of unit squares located at the boundary of the diagrams.
The grandcanonical ensembles, which were introduced by Vershik \cite{V},
are uniform measures under conditioning on their size (or equivalently, area).
We then show that, as the averaged size of the diagrams diverges, the 
corresponding height variable converges to a solution of a 
certain non-linear partial differential equation under a proper hydrodynamic
scaling. Furthermore, the stationary solution of the limit equation 
is identified with the so-called Vershik curve.  We discuss
both uniform and restricted uniform statistics for the Young diagrams.
\footnote{
Graduate School of Mathematical Sciences,
The University of Tokyo, Komaba, Tokyo 153-8914, Japan.}
\footnote{ 
$\quad$ e-mail: funaki@ms.u-tokyo.ac.jp and sasada@ms.u-tokyo.ac.jp, 
Fax: +81-3-5465-7011.}
\footnote{
\textit{Abbreviated title $($running head$)$: Hydrodynamic limit for 
2D Young diagrams.}}
\footnote{
\textit{Supported in part by the JSPS Grants $($A$)$ 18204007
and 21654021.}}
\end{abstract}

\section{Introduction}

The asymptotic shapes of two-dimensional random Young diagrams with large size
were studied by Vershik \cite{V} under several types of statistics including 
the uniform and restricted uniform statistics, which were also called the
Bose and Fermi statistics, respectively.   To each partition 
$p=\{p_1\ge p_2\ge\cdots\ge p_j\ge 1\}$ of a positive integer $n$ by positive
integers $\{p_i\}_{i=1}^j$ (i.e., $n= \sum_{i=1}^j p_i$), a Young diagram is
associated by piling up $j$ sticks of height $1$ and side-length $p_i$,
more precisely, the height function of the Young diagram is defined by
\begin{equation} \label{eq:1.1.1}
\psi_p(u) = \sum_{i=1}^j 1_{\{u<p_i\}}, \quad u\ge 0.
\end{equation}
The closure of the interior of its ordinate set is called the Young diagram
of the partition $p$. Note that, in most literatures, the figures of
Young diagrams are upside-down compared with the graph defined by
\eqref{eq:1.1.1}.

For each fixed $n$, the uniform statistics (U-statistics in short)
$\mu_U^n$ assigns an equal probability to each of possible partitions
$p$ of $n$, i.e., to the Young diagrams of area $n$.  The restricted uniform
statistics (RU-statistics in short) $\mu_R^n$ also assigns an equal 
probability, but restricting to the distinct partitions satisfying
$q= \{q_1>q_2>\cdots>q_j\ge 1\}$.
These probabilities are called canonical ensembles.  Grandcanonical
ensembles $\mu_U^\e$ and $\mu_R^\e$ with parameter $0<\e<1$
are defined by superposing the canonical ensembles in a similar
manner known in statistical physics, see \eqref{eq:2.gcB} and
\eqref{eq:2.gcF} below.  Vershik \cite{V} proved that, under the canonical
U- and RU-statistics $\mu_U^{N^2}$ and $\mu_R^{N^2}$ (with $n=N^2$),
the law of large numbers holds as $N\to\infty$ for the scaled height variable 
\begin{equation} \label{eq:scale}
\tilde{\psi}_p^N(u):= \frac1N \psi_p(Nu), \quad u\ge 0,
\end{equation}
of the Young diagrams $\psi_p(u)$ 
with size (i.e., area) $N^2$ and for $\tilde{\psi}_q^N(u)$ defined similarly,
and the limit shapes $\psi_U$ and $\psi_R$ are given by 
\begin{equation} \label{eq:1.1}
\psi_U(u)=-\frac{1}{\a} \log\big(1-e^{-\a u}\big)
\quad \text{ and } \quad
\psi_R(u)=\frac{1}{\b} \log\big(1+e^{-\b u}\big), \quad u\ge 0,
\end{equation}
with $\a= \pi/\sqrt{6}$ and $\b= \pi/\sqrt{12}$, respectively.
These results can be extended to the corresponding grandcanonical ensembles
$\mu_U^\e$ and $\mu_R^\e$, if the averaged size
of the diagrams is $N^2$ under these measures.  Such types of results are
usually called the equivalence of ensembles in the context of statistical
physics.  The corresponding central limit theorem and large deviation
principle (under canonical ensembles)
were shown by Pittel \cite{P} and Dembo et.\ al.\ \cite{DVZ},
respectively. All these results are at the static level.

The purpose of this paper is to study and extend these results from a
dynamical point of view.  We will see that, to the grandcanonical U- and
RU-statistics, one can associate a weakly asymmetric zero-range
process $p_t$ respectively a weakly asymmetric simple exclusion process
$q_t$ on a set of positive integers with a stochastic reservoir at the
boundary site $\{0\}$ in both processes 
as natural time evolutions of the Young diagrams, or more precisely, 
those of the gradients of their height functions. Then, under 
the diffusive scaling in space and time and choosing the parameter 
$\e =\e(N)$ of the grandcanonical ensembles such that the averaged size of
the Young diagrams is $N^2$, we will derive the hydrodynamic equations in
the limit and show that the Vershik curves defined by \eqref{eq:1.1} 
are actually unique stationary solutions to the limiting non-linear
partial differential equations in both cases.

In Section 2, after defining the ensembles and the corresponding dynamics,
we formulate our main theorems, see Theorems \ref{thm:bthm} and 
\ref{thm:fthm}.  In Section 3, we study the asymptotic behaviors of $\e(N)$
as $N\to\infty$. The weakly asymmetric zero-range process $p_t$ with a
stochastic reservoir at the boundary $\{0\}$ can be transformed into
the weakly asymmetric simple exclusion process $\bar{\eta}_t$ on $\Z$
without any boundary condition.  In Section 4, we study such transformations
and also those for the limit equations, and give the proof of the main 
theorem for the U-case (i.e.\ the case corresponding to the U-statistics).
The hydrodynamic limit for $\bar{\eta}_t$ is
indeed already known \cite{G}, and we apply this result for $\bar{\eta}_t$.
The idea of transforming $p_t$ into $\bar{\eta}_t$,
which is indeed known in the study of particle systems, is useful to avoid 
the difficulty in treating singularities at the boundary $u=0$, which
appear in the limit of $\tilde{\psi}_p^N(u)$.  The main theorem for the
RU-case (i.e.\ the case corresponding to the RU-statistics)
is proved in Section 5.  Our method is to apply the Hopf-Cole
transformation for the microscopic process $q_t$, which was originally
introduced by G\"{a}rtner \cite{G}.  
This transformation linearizes the leading term
in the time evolution $q_t$ even at the microscopic level so that 
one can avoid to show the one-block and two blocks' estimates, which are
usually required in the procedure establishing the hydrodynamic limit.
The only task left is to study the boundary behavior of the transformed
process, but a rather simple argument leads to the desired ergodic
property of our process at the boundary, see Lemma \ref{lem:ergod}
below.  

The corresponding dynamic fluctuations will be discussed in a separate 
paper \cite{FS-1}.  Our dynamics can be interpreted as evolutional models of
(non-increasing) interfaces which separate $\pm$-phases in a zero-temperature
two-dimensional Ising model defined on a first quadrant, see Spohn \cite{Sp}
and Remark \ref{rem:2.2} below.  A randomly growing Young diagram was studied
by Johansson \cite{J1}, \cite{J2} in relation to random matrices.  See 
\cite[Section 16.4]{Fu} for a quick review of some related results.

\section{Two-dimensional Young diagrams and main results}

In this section, for U- and RU-statistics individually, we define
the grandcanonical and canonical ensembles, introduce the corresponding
dynamics and then formulate the main results concerning the space-time
scaling limits for them.  Throughout the paper, we will use the
following notation: $\Z_+=\{0,1,2, \ldots\}$, $\N=\{1,2,3, \ldots\}$,
$\R_+=[0,\infty)$ and $\R_+^\circ=(0,\infty)$.

\subsection{U-statistics}
For each $n\in \N$, we denote by $\mathcal{P}_n$ the set of all
partitions of $n$ into positive integers, that is, the set of all
$p=(p_i)_{i\in\N}$ satisfying $p_1 \ge p_2 \ge \cdots \ge p_i
\ge \cdots$, $p_i \in \Z_+$ and $\sum_{i\in\N} p_i =n$. 
For $n=0$, we define $\mathcal{P}_0=\{0\}$, where $0$ is a sequence
such that $p_i=0$ for all $i\in \N$.
We consider $p$ as an infinite sequence by adding infinitely many
$0$'s rather than a finite sequence as in Section 1.  This will be
convenient from the point of view of the corresponding particle system.
The union of $\mathcal{P}_n$ is denoted by $\mathcal{P}$: $\mathcal{P}
=\cup_{n\in\Z_+} \mathcal{P}_n$.  The sum of $p_i$'s in
$p \in \mathcal{P}$ is described as $n(p)$: $n(p)=\sum_{i\in\N} p_i$,
and called the size or area of $p$.

For $p \in \mathcal{P}$, we assign a right continuous non-increasing 
step-function $\psi_p$ on $\R_+$ called the height function as follows:
\begin{equation} \label{eq:2.1}
\psi_p(u)=\sum_{i\in \N} 1_{\{u < p_i\}}, \quad u\in \R_+.
\end{equation}
In particular, we always have $\int^{\infty}_0 \psi_p(u)du = n(p)$.

For $0 < \e <1$, let $\mu^{\e}_U$ be the probability measure on $\mathcal{P}$
determined by
\begin{equation} \label{eq:2.gcB}
\mu^{\e}_U(p)=\frac{1}{Z_U(\e)}\e^{n(p)}, \quad p \in \mathcal{P},
\end{equation}
where $Z_U(\e)=\prod^{\infty}_{k=1} (1-\e^k)^{-1} \big( = \sum_{n=0}^\infty
p(n)\e^n, p(n)=\sharp \mathcal{P}_n\big)$ is the normalizing
constant.  The measure $\mu^{\e}_U$ has the property
$\mu^{\e}_U|_{\mathcal{P}_n}(p)=\mu^n_U(p), p \in \mathcal{P},$
where $\mu^{\e}_U|_{\mathcal{P}_n}$ stands for the conditional probability
of $\mu^{\e}_U$ on $\mathcal{P}_n$ and $\mu^n_U$ is the uniform probability
measure on  $\mathcal{P}_n$.  The measures
$\mu^{\e}_U$ and $\mu^n_U$ play similar roles to the grandcanonical and
canonical ensembles in statistical physics, respectively.

Now, we construct dynamics of two-dimensional Young diagrams, which
have $\mu^{\e}_U$ as their invariant measures. Let $p_t \equiv p^{\e}_t
=(p_i(t))_{i\in\N}$ be the Markov process on $\mathcal{P}$ defined by means 
of the infinitesimal generator $L_{\e,U}$ acting on functions 
$f:\mathcal{P} \to \R$ as
\begin{equation}  \label{eq:2.2}
L_{\e,U}f(p) = \sum_{i \in \N} \big[ \e 1_{\{p_{i-1} > p_i\}}
\{f(p^{i,+})-f(p)\}+1_{\{p_i > p_{i+1}\}}\{f(p^{i,-})-f(p)\} \big],
\end{equation}
where $p^{i,\pm}=(p_j^{i,\pm})_{j\in \N} \in \mathcal{P}$ are defined by
\begin{equation}  \label{eq:2.3}
 p^{i,\pm}_j = \begin{cases}
	p_j  & \text{if $j \neq i$}, \\
	p_i \pm 1 & \text{if $j=i$}. 
\end{cases} 
\end{equation}
In \eqref{eq:2.2}, we regard $p_0 = \infty$. Note that $n(p_t)$
and $\frak{n}(p_t) :=\sharp\{i\in \N; p_i(t)\ge 1\}$ change in time
but always stay finite.  It is easy to see that $\mu^{\e}_U$ is invariant
under such dynamics for every $0<\e<1$ by showing that 
$\sum_{p\in\mathcal{P}} L_{\e,U}f(p) \mu_U^\e(p) =0$ for a sufficiently wide
class of functions $f$.  We will think of $p_i(t)$ as the position of the
$i$th particle.  The total number of particles $\frak{n}(p_t)$ on the
region $\N$ changes only through the creation and annihilation
of particles at the boundary site $\{0\}$.  In fact, the first part
in the sum \eqref{eq:2.2} with $i= \frak{n}(p)+1$ represents that
a new particle is provided from the boundary site $\{0\}$ to the site 
$\{1\}$ with rate $\e$, while the second part with $i= \frak{n}(p)$
indicates that a particle at $\{1\}$ jumps to $\{0\}$ and disappears
with rate $1$.  In other words, a stochastic reservoir is located at
the boundary site $\{0\}$ of $\N$.

For a probability measure $\nu$ on $\mathcal{P}$ and $N\ge 1$, we denote by $\mathbb{P}_{\nu}^N$ the distribution on the path space $D(\R_+,\mathcal{P})$ of the process $p_t\equiv p_t^N$ with generator $N^2 L_{\e(N),U}$, which is accelerated by the factor $N^2$ and the initial measure $\nu$. Here, $\e(N)$ is defined 
by the relation:
\begin{equation}  \label{eq:2.eB}
E_{\mu^{\e(N)}_U}[n(p)]=N^2.
\end{equation}
Let $X_U$ be the function space defined by
\begin{align*}
X_U:= \{\psi:\R_+^\circ \to \R_+^\circ; \psi \in C^1, \psi^{\prime}<0, 
\lim_{u \downarrow 0}\psi(u)=\infty, \lim_{u \uparrow \infty}\psi(u)=0\},
\end{align*}
where $\psi'=d\psi/du$.  With these notations our first main theorem is stated
as follows.  Recall that the scaled height variable $\tilde{\psi}_p^N(u)$ is
defined by \eqref{eq:scale} for $p\in\mathcal{P}$.

\begin{thm}\label{thm:bthm}
Let $(\nu^N)_{N \ge 1}$ be a sequence of probability measures on 
$\mathcal{P}$ such that 
\begin{equation}  \label{eq:2.ini}
\lim_{N \to \infty} \nu^N[ \sup_{u\in[u_0,u_1]} |\tilde{\psi}_p^N(u)
- \psi_0(u)|>\de]=0
\end{equation}
holds for every $\de>0$, $0<u_0<u_1$ and some function $\psi_0 \in
X_U$. Then, for every $t>0$,
\begin{equation}  \label{eq:2.t}
\lim_{N \to \infty} \mathbb{P}_{\nu^N}^N[|\int^{\infty}_0 f(u) 
\tilde{\psi}_{p_t}^N(u)du - \int^{\infty}_0 f(u)\psi(t,u)du|>\de]=0
\end{equation}
holds for every $\de>0$ and $f \in C_0(\R_+^\circ)$, where $C_0(\R_+^\circ)$
is the class of all functions $f\in C(\R_+^\circ)$ having compact supports in 
$\R_+^\circ$ and $\psi(t,u)$ is the unique classical solution (in the space
$X_U$) of the non-linear partial differential equation (PDE):
\begin{equation}
\left\{
\begin{aligned}
\partial_t\psi & = \partial_u \left(\frac{\partial_u\psi}{1-\partial_u\psi}\right) + \a \frac{\partial_u\psi}{1-\partial_u\psi}, \\ \label{eq:bhydro}
\psi(0,\cdot) & = \psi_0(\cdot), \\
\psi(t,\cdot) & \in X_U, \quad t \ge 0,
\end{aligned}
\right.
\end{equation}
where $\partial_t\psi=\partial\psi/\partial t$,
$\partial_u\psi=\partial\psi/\partial u$
and $\a= \pi/\sqrt{6}$.
\end{thm}

\begin{rem}
The function $\psi_U$ defined in \eqref{eq:1.1} is the unique stationary 
solution in the class $X_U$ of the equation (\ref{eq:bhydro}).  The curve
in the first quadrant of $xy$-plane determined by the equation $y=\psi_U(x)$
is called the Vershik curve (in U-statistics).
\end{rem}

In this way, the derivation of the Vershik curve is understandable from 
the dynamical point of view.

\begin{rem} \label{rem:2.2}
Spohn discussed in \cite[Appendix A]{Sp} two-dimensional interfacial
dynamics, motivated by the zero-temperature Ising model, under the
periodic boundary condition with symmetric jump rates and derived the
non-linear PDE \eqref{eq:bhydro} with $\a=0$ under the hydrodynamic
scaling limit.  He studied interfaces having graphical
representations as in our setting, but not necessarily being monotone.
See \cite[Section 4]{CS}, \cite{CSS}, \cite{CK} for further studies.
Shlosman \cite{Sh} discussed the similarity between the approach
from the Young diagrams and the Wulff problem in the Ising model.
Aldous and Diaconis \cite{AD} used an idea of the hydrodynamic limit 
to give a \lq\lq soft" proof for the asymptotic behavior of the length
of the longest increasing subsequence of random permutations.
\end{rem}

\begin{rem}
The large deviation rate function $I(\psi)$ under the canonical ensemble 
of U-statistics  $\mu_U^n$ is described in Theorem 1.2 of \cite{DVZ}.
We can compute its functional derivative and find that it is given by the
formula:
$$
\frac{\de I}{\de\psi(u)} = \frac{\psi''(u)+\a\psi'(u)(1-\psi'(u))}
{\psi'(u)(1-\psi'(u))}.
$$
On the other hand, the right hand side of our hydrodynamic equation
\eqref{eq:bhydro} is equal to
$$
\frac{\psi''(u)+\a\psi'(u)(1-\psi'(u))}{(1-\psi'(u))^2}.
$$
These formulas have similarity but are not exactly the same.  Recall
that we discuss the dynamics associated with the grandcanonical ensemble.
The dynamics for the canonical ensemble involve much complexity.
\end{rem}

\subsection{RU-statistics}
Denote by $\mathcal{Q}_n$ the set of all partitions of $n\in\N$
into distinct positive integers, that is, the set of all 
$q=(q_i)_{i\in \N} \in \mathcal{P}_n$ satisfying $q_i > q_{i+1}$
if $q_i >0$.  The union of $\mathcal{Q}_n$ is denoted by $\mathcal{Q}$:
$\mathcal{Q}=\cup_{n\in\Z_+} \mathcal{Q}_n$, where $\mathcal{Q}_0
= \mathcal{P}_0$. Let $n(q)$ be the sum of $q_i$'s in 
$q \in \mathcal{Q}$.  The function $\psi_q$ on $\R_+$ is assigned to
$q \in \mathcal{Q}$ by the relation \eqref{eq:2.1}.

For $0 < \e <1$, let $\mu^{\e}_R$ be the probability measure on $\mathcal{Q}$
determined by
\begin{equation} \label{eq:2.gcF}
\mu^{\e}_R(q)=\frac{1}{Z_R(\e)}\e^{n(q)}, \quad q \in \mathcal{Q},
\end{equation}
where $Z_R(\e)=\prod^{\infty}_{k=1} (1+\e^k) \big( = \sum_{n=0}^\infty
p_{\not=}(n)\e^n, p_{\not=}(n)=\sharp \mathcal{Q}_n\big)$ is the normalizing
constant.  The conditional measure $\mu^{\e}_R|_{\mathcal{Q}_n}$ of
$\mu^{\e}_R$ on $\mathcal{Q}_n$ coincides with the uniform probability measure 
$\mu^n_R$ on $\mathcal{Q}_n$.  The measures $\mu^{\e}_R$ and $\mu^n_R$
are the grandcanonical and canonical ensembles in the RU-statistics,
respectively.

Now, we construct the dynamics associated with $\mu^{\e}_R$. 
Let $q_t \equiv q^{\e}_t =(q_i(t))_{i\in\N}$ be the Markov process on 
$\mathcal{Q}$ with the infinitesimal generator $L_{\e,R}$ acting on functions 
$f:\mathcal{Q} \to \R$ as
\begin{equation}  \label{eq:2.10}
L_{\e,R}f(q) = \sum_{i \in \N} \big[\e 1_{\{q_{i-1} > q_i + 1\}}
\{f(q^{i,+})-f(q)\}+ 1_{\{q_i > q_{i+1} + 1 \text{ or } q_i = 1\}}
\{f(q^{i,-})-f(q)\} \big],
\end{equation}
where $q^{i,\pm}\in\mathcal{Q}$ are defined by the formula \eqref{eq:2.3}
and we regard $q_0 =\infty$. It is easy to see that $\mu^{\e}_R$ is invariant
under such dynamics.  Similarly to the U-case, the model defined by the
generator \eqref{eq:2.10} involves a stochastic reservoir at $\{0\}$.
The only difference is that the creation of a new particle at $\{1\}$
is allowed if this site is vacant.

For a probability measure $\nu$ on $\mathcal{Q}$ and $N\ge 1$, we denote by $\mathbb{Q}_{\nu}^N$ the distribution on the path space $D(\R_+,\mathcal{Q})$ of the process $q_t \equiv q_t^N$ with generator $N^2 L_{\e(N),R}$ and the initial measure $\nu$. Here, $\e(N)$ is defined by the relation:
\begin{equation}  \label{eq:2.eF}
E_{\mu^{\e(N)}_R}[n(p)]=N^2.
\end{equation}
Let $X_R$ be the function space defined by
$$
X_R:= \{\psi:\R_+ \to \R_+; \psi \in C^1,
-1 \le \psi^{\prime} \le 0, \psi'(0)=-1/2, 
\lim_{u \uparrow \infty}\psi(u)=0\}.
$$
Our second main theorem is stated as follows.  The scaled height variable
$\tilde{\psi}^N_{q}(u)$ is defined by \eqref{eq:scale} for $q\in \mathcal{Q}$.

\begin{thm}\label{thm:fthm}
Let $(\nu^N)_{N \ge 1}$ be a sequence of probability measures on $\mathcal{Q}$ such that 
\begin{equation}\label{eq:2.fini}
\lim_{N \to \infty} \nu^N[|\int^{\infty}_0 f(u) \tilde{\psi}^N_{q}(u)du- \int^{\infty}_0 f(u)\psi_0(u)du|>\de]=0
\end{equation}
holds for every $\de>0$, $f \in C_0(\R_+^\circ)$ and some function $\psi_0
\in X_R$. Then, for every $t>0$,
\begin{equation}\label{eq:2.ft}
\lim_{N \to \infty} \mathbb{Q}_{\nu^N}^N[|\int^{\infty}_0 f(u) 
\tilde{\psi}^N_{q_t}(u)du - \int^{\infty}_0 f(u)\psi(t,u)du|>\de]=0
\end{equation}
holds for every $\de>0$ and $f \in C_0(\R_+^\circ)$, where $\psi(t,u)$ is the
unique classical solution (in the space $X_R$) of the non-linear partial 
differential equation:
\begin{equation}
\left\{
\begin{aligned}
\partial_t\psi &= \partial_u^2 \psi + \b \ \partial_u\psi (1+\partial_u\psi),\\
\psi(0,\cdot) & = \psi_0(\cdot), \\
\psi(t,\cdot) & \in X_R, \quad t \ge 0, 
    \label{eq:fhydro}
\end{aligned}
\right.
\end{equation}
and $\b=\pi/\sqrt{12}$.
\end{thm}

\begin{rem}
The function $\psi_R$ defined in \eqref{eq:1.1} is the unique stationary
solution in the class $X_R$ of the equation (\ref{eq:fhydro}).  The curve
determined by the equation $y=\psi_R(x)$ is called the Vershik curve
(in RU-statistics).
\end{rem}

\section{Asymptotic behaviors of $\e(N)$}

Before giving the proof of our main theorems, we study in this section
the asymptotic behaviors of $\e(N)$ defined by
\eqref{eq:2.eB} and \eqref{eq:2.eF} in U- and RU-statistics,
respectively, as $N\to\infty$.

\subsection{U-statistics}

Let $\e(N)$ be defined by the relation \eqref{eq:2.eB}.

\begin{lem}\label{lem:bexp}
We have that $\e(N)=1- \a/N + O(\log N/N^2)$ as
$N\to\infty$.
\end{lem}
\begin{proof}
First, we calculate the expected value of the size $n(p)$ of $p\in\mathcal{P}$
under the probability measure $\mu^{\e}_U$.  In fact,
\begin{align*}
E_{\mu^{\e}_U}[n(p)] &=\frac{1}{Z_U(\e)}\sum_{p}n(p)\e^{n(p)} 
	=\e \big(\log{Z_U(\e)}\big)^{\prime} \\
	&=\sum_{k=1}^{\infty} \frac{k \e^k}{1-\e^k}
	=\sum_{k=1}^{\infty}\sum_{m=1}^{\infty}k \e^{mk}
	= \sum_{m=1}^{\infty}\frac{\e^{m}}{(1-\e^m)^2}.
\end{align*}
The last equality follows from the simple identity
$\sum_{k=1}^{\infty}k z^{k}=z/(1-z)^2$ for $0 \le z <1$.
However, the inequality of arithmetic and geometric means and 
some simple estimations prove
\begin{displaymath}
\frac{1}{m} \le \frac{1}{1+\e + \e^2 + \cdots + \e^{m-1}} \le 
\frac{\e^{(-m+1)/2}}{m},
\end{displaymath}
and thus, recalling $\a^2 = \pi^2/6$ and $\e<1$, we have that
\begin{equation*} %% \label{eq:3.1-c}
\frac1{(1-\e)^2} \sum_{m=1}^{\infty}\frac{\e^{m}}{m^2}
\le E_{\mu^{\e}_U}[n(p)]
\le \frac1{(1-\e)^2} \sum_{m=1}^{\infty}\frac{\e}{m^2} < \frac{\a^2}{(1-\e)^2}.
\end{equation*}
Therefore, by \eqref{eq:2.eB}, we have for $\e=\e(N)$
\begin{equation} \label{eq:3.2-c}
0< \a^2- (1-\e)^2N^2
\le \a^2 - \sum_{m=1}^{\infty}\frac{\e^{m}}{m^2}
= \sum_{m=1}^{\infty}\frac{1-\e^{m}}{m^2}.
\end{equation}
Since the right hand side tends to $0$ as $\e\uparrow 1$ (or as $N\to\infty$),
we see that $(1-\e)N$ tends to $\a$ as $N\to\infty$ which implies that 
$\e \equiv \e(N) = 1-\a/N +o(1/N)$.  To derive more precise estimate for 
the error term, we will show that the right hand side of \eqref{eq:3.2-c}
admits a bound:
\begin{equation} \label{eq:3.3-c}
\sum_{m=1}^{\infty}\frac{1-\e^{m}}{m^2} \le \frac{C\log N}N,
\end{equation}
with some $C>0$.  Indeed, once this is shown, the proof of the lemma
is concluded.  To prove \eqref{eq:3.3-c}, noting that the function
$f_\e(x) := (1-\e^x)/x^2, x>0,$ is non-increasing, we have that
\begin{align*}
\sum_{m=1}^{\infty}\frac{1-\e^{m}}{m^2}
& \le f_\e(1) + \int_1^\infty f_\e(x)dx
= f_\e(1) -\log\e \int_{-\log\e}^\infty \frac{1-e^{-y}}{y^2} dy  \\
& \le (1-\e) -\log\e \big(-\log(-\log\e)\big)  -\log\e,
\end{align*}
where the last inequality follows by dividing the integral over 
$[-\log\e,\infty)$ into the sum of those over $[-\log\e,1]$ and $[1,\infty)$
and then by estimating the integrands by $1/y$ and $1/y^2$, respectively.
This implies \eqref{eq:3.3-c} by recalling $\e= 1-\a/N +o(1/N)$.
\end{proof}

\begin{rem}
A rude version of Hardy-Ramanujan's formula: $p(n) =\sharp\mathcal{P}_n
\sim e^{\sqrt{2/3}\pi\sqrt{n}}$ as $n\to\infty$ implies that
$$
E_{\mu^{\e}_U}[n(p)] \sim \sum_{n=1}^\infty n
\frac{e^{\sqrt{2/3}\pi\sqrt{n}- (\log \e^{-1})n}}{Z_U(\e)}.
$$
Since the function $f(x):= \sqrt{2/3}\pi\sqrt{x}- (\log \e^{-1})x, x>0$
attains its maximal value at $x(\e) = \big(\pi/(\sqrt{6}\log\e^{-1})\big)^2$,
we see that $E_{\mu^{\e}_U}[n(p)]$ behaves as $x(\e)$ as $\e\uparrow 1$ and
this shows that $\e(N) \sim 1-\a/N$ as $N\to\infty$.
\end{rem}

\subsection{RU-statistics}

Let $\e(N)$ be defined by the relation \eqref{eq:2.eF}.

\begin{lem}  \label{lem:3.2}
We have that $\e(N)=1- \b/N + O(\log N/N^2)$
as $N\to\infty$.
\end{lem}
\begin{proof}
First, we calculate the expected value of $n(p)$ under $\mu^{\e}_R$:
\begin{align*}
E_{\mu^{\e}_R}[n(p)] &=\frac{1}{Z_R(\e)}\sum_{p}n(p)\e^{n(p)} 
	=\e \big(\log{Z_R(\e)}\big)^{\prime} \\
	&=\sum_{k=1}^{\infty} \frac{k \e^k}{1+\e^k}
	=\sum_{k=1}^{\infty}\sum_{m=1}^{\infty}k (-1)^{m-1} \e^{mk}.
\end{align*}
Thus, similarly to the U-case, we have that
\begin{equation} \label{eq:f}
E_{\mu^{\e}_R}[n(p)]=\sum_{m=1}^{\infty}(-1)^{m-1}\frac{\e^{m}}{(1-\e^m)^2}
= \Si_{\rm o} (\e)- \Si_{\rm e}(\e),
\end{equation}
where $\Si_{\rm o}(\e)$ and $\Si_{\rm e}(\e)$ are the sums taken over odd and 
even numbers, respectively, i.e.\
\begin{displaymath}
\Si_{\rm o} (\e):=\sum_{m=1}^{\infty} \frac{\e^{2m-1}}{(1-\e^{2m-1})^2},  \quad
\Si_{\rm e} (\e):= \sum_{m=1}^{\infty}\frac{\e^{2m}}{(1-\e^{2m})^2}.
\end{displaymath}
Note that one can change the order of the sum in \eqref{eq:f} since the 
series converges absolutely.  Now recalling $\b^2 = \a^2/2$, by 
\eqref{eq:2.eF}, we have for $\e=\e(N)$
\begin{align*}
\b^2 - (1-\e)^2N^2 = \big\{\a^2 - (1-\e)^2 (\Si_{\rm o}(\e) + \Si_{\rm e}(\e))
\big\} - \big\{ \a^2- 4 (1-\e)^2 \Si_{\rm e}(\e)\big\}/2.
\end{align*}
However, since $\Si_{\rm e}(\e)=\Si_{\rm o}(\e^2) + \Si_{\rm e}(\e^2) $ and $(1+\e)^2 \le 4$,
we have that
\begin{align*}
|\b^2 - (1-\e)^2N^2|
& \le \left| \a^2 - (1-\e)^2 (\Si_{\rm o}(\e) + \Si_{\rm e}(\e)) \right| +
\left| \a^2 - (1-\e^2)^2 (\Si_{\rm o}(\e^2) + \Si_{\rm e}(\e^2)) \right|/2\\
& \le \left\{(1-\e) -\log\e \big(-\log(-\log\e)\big) -\log\e\right\} \\
& \quad+ \left\{(1-\e^2) - 2 \log\e \big(-\log(-2\log\e)\big) 
  -2\log\e\right\}/2.
\end{align*}
The second inequality is shown in the proof of Lemma \ref{lem:bexp}.  
This first implies that $\e\equiv \e(N) = 1-\b/N + o(1)$ and then completes
the proof of the lemma as in the last part of the proof of 
Lemma \ref{lem:bexp}.
\end{proof}

The precise error estimates in Lemmas \ref{lem:bexp} and \ref{lem:3.2}
are only needed in \cite{FS-1}, see Remark \ref{rem:5.1} below.

\section{Proof of Theorem \ref{thm:bthm}}

This section gives the proof of Theorem \ref{thm:bthm} for the U-case.
In the process $p_t$, the particles are distinguished from each other
and numbered from the right.  However, if we are only concerned with the
number of particles at each site and define $\xi_t =
(\xi_t(x))_{x\in \Z_+}$ by $\xi_t(x) = \sharp\{i; p_i(t)=x\} \in 
\Z_+$ for $x\in \N$ and $\xi_t(0)=\infty$, then $\xi_t$ becomes the
weakly asymmetric zero-range process on $\N$ with the weakly asymmetric
stochastic reservoir at $\{0\}$.  We can think of $\xi_t(x)$ as the 
(negative) gradient of the height function $\psi_{p_t}$ at $u=x$ in the
sense that $\xi_t(x) = \psi_{p_t}(x-1)-\psi_{p_t}(x)$.

Actually, the stochastic reservoir for $p_t$ or $\xi_t$ located at $\{0\}$
can be removed under a simple transformation.  Indeed, we transform the 
process $p_t$ into another process $\bar{\eta}_t$ on $\Z$, which is roughly
defined as follows: With each $p\in\mathcal{P}$, we associate a family of
particles located at $(i,p_i)$ in the $xy$-plane and project them 
perpendicularly to the line $\{y=-x\}$ rescaled by $\sqrt{2}$.  Or, one can
say that we first rotate the $xy$-plane by 45 degree to the left-handed
direction and then project the particles to the $x$-axis rescaled by 
$\sqrt{2}$. This determines a configuration $\bar{\eta}$ on $\Z$.  Such
transformation is sometimes used in the study of particle systems.
As we will see, in the RU-case, one can not find this kind of nice
transformation which removes the stochastic reservoir.

\subsection{Transformation for the process $p_t$}

We introduce a transformation of our process $p_t$ on $\N$ to a weakly 
asymmetric simple exclusion process $\bar{\eta}_t$ on $\Z$ mentioned above.
Denote by $\chi_U$ the state space of the transformed process:
\begin{equation*}
\chi_U:=\{\bar{\eta} \in \{0,1\}^{\Z}; \sum_{x \le 0}(1-\bar{\eta}(x))
=\sum_{x \ge 1}\bar{\eta}(x) \ < \ \infty\}.
\end{equation*}
In particular, if $\bar{\eta}\in\chi_U$, then there exist $x_\pm\in\Z$ 
such that $\bar{\eta}(x) =1$ for all $x\le x_-$ and $\bar{\eta}(x)=0$ for
all $x\ge x_+$.  For $\bar{\eta} \in \chi_U$, we assign two functions 
$\zeta^{-}_{\bar{\eta}}$ and $\zeta^{+}_{\bar{\eta}}$ on $\Z$ by the 
following rule:
\begin{equation} \label{eq:4.1-a}
\zeta^{-}_{\bar{\eta}}(x)=\sum_{z \le x}(1-\bar{\eta}(z))
\quad \text{ and } \quad 
\zeta^{+}_{\bar{\eta}}(x)=\sum_{z \ge x+1}\bar{\eta}(z), \quad x\in \Z.
\end{equation}
By definition, $\zeta^-_{\bar{\eta}}$ and $\zeta^{+}_{\bar{\eta}}$ are 
monotone non-negative integer-valued functions. Now, we construct one-to-one 
correspondence between $\chi_U$ and $\mathcal{P}$. For 
$\bar{\eta} \in \chi_U$, we assign $p^{\bar{\eta}}=
(p_i^{\bar{\eta}})_{i\in\N}\in \mathcal{P}$ by the following rule:
\begin{displaymath}
p^{\bar{\eta}}_i=\zeta^-_{\bar{\eta}}(x_i), \quad i \in \N,
\end{displaymath}
where $x_i$ is the unique element of $\Z$ which satisfies 
$\zeta^{+}_{\bar{\eta}}(x_i-1)=i$ and $\zeta^{+}_{\bar{\eta}}(x_i)=i-1$.
In other words, the family $\{x_i\}_{i\in\N}$ is determined by numbering
the set $\{x\in\Z;\bar{\eta}(x)=1\}$ by $i\in\N$ from 
the right and $p_i^{\bar{\eta}}= \sharp\{x\le x_i; \bar{\eta}(x)=0\}$.
We can show that the map $\bar{\eta} \to p^{\bar{\eta}}$ is well-defined and
also it is a bijection from $\chi_U$ to $\mathcal{P}$. So we denote its 
inverse map by $p \to \bar{\eta}^p$.  Note that the origin $0$ is determined
by the condition $\zeta^-_{\bar{\eta}}(0)=\zeta^{+}_{\bar{\eta}}(0)$ or
equivalently $\sharp\{x\le 0; \bar{\eta}(x)=0\} = \sharp\{x\ge 1; 
\bar{\eta}(x)=1\}$, i.e., the number of empty sites on the left to the
origin is equal to that of particles on the right to the site $1$.

We now consider the Markov process $\bar{\eta}_t$ on $\chi_U$ with the 
generator $\bar{L}_{\e,U}$ acting on functions $f:\chi_U \to \R$ as
\begin{displaymath}
\bar{L}_{\e,U}f(\bar{\eta}) = \sum_{x \in \Z} \big[\e c_+(x,\bar{\eta})
+ c_-(x,\bar{\eta}) \big] \{f(\bar{\eta}^{x,x+1})-f(\bar{\eta})\},
\end{displaymath}
where 
\begin{equation}  \label{eq:4.c-jump}
c_+(x,\bar{\eta}) = 1_{\{\bar{\eta}(x)=1,\bar{\eta}(x+1)=0\}},
\quad
c_-(x,\bar{\eta}) = 1_{\{\bar{\eta}(x)=0,\bar{\eta}(x+1)=1\}},
\end{equation}
and
\begin{equation}  \label{eq:4.eta}
 \bar{\eta}^{x,y}(z) = \begin{cases}
	\bar{\eta}(z)  & \text{if $z \neq x,y$}, \\
	\bar{\eta}(y)  & \text{if $z=x$}, \\ 
	\bar{\eta}(x)  & \text{if $z=y$}. 
\end{cases} 
\end{equation}
Note that the relation $\zeta^-_{\bar{\eta}}(0)
=\zeta^+_{\bar{\eta}}(0)$
is invariant under the transition from $\bar{\eta}$ to $\bar{\eta}^{x,x+1}$ 
for all $x\in\Z$. The following lemma is easy so that the proof is omitted.

\begin{lem}
Two processes $\{p_t\}_{t\ge 0}$ and $\{p^{\bar{\eta}_t}\}_{t\ge 0}$ have
the same distributions on the path space $D(\R_+,\mathcal{P})$.
\end{lem}

For a probability measure $\nu$ on $\chi_U$ and $N\ge 1$, we denote by 
$\bar{\mathbb{P}}_{\nu}^N$ the distribution on the path space $D(\R_+,\chi_U)$
of the process $\bar{\eta}_t^N$ with generator $N^2 \bar{L}_{\e(N),U}$ and 
the initial measure $\nu$, where $\e(N)$ is defined by \eqref{eq:2.eB}.
By Lemma \ref{lem:bexp}, since $\e(N)$ is close to $1$ for large $N$, we can
think of the process $\bar{\eta}_t^N$ as a weakly asymmetric simple exclusion
process on $\Z$.  The hydrodynamic limit of such process is already known.  
Indeed, let $Y_U$ be the function space defined by
\begin{align*}
Y_U:= \{\rho:\R \to (0,1); \rho \text{ is continuous}, \int^0_{-\infty}
(1-\rho(v))dv=\int^{\infty}_0 \rho(v)dv < \infty \}.
\end{align*}
Then, for the scaled empirical measures of the process $\bar{\eta}_t^N$ 
defined by
\begin{equation} \label{eq:4.pi}
\pi_t^N(dv) = \frac1N \sum_{x\in \Z}\bar{\eta}_t^N(x) \de_{x/N}(dv),
 \quad t\ge 0, \; v \in \R,
\end{equation}
we have the following proposition, see G\"{a}rtner \cite{G}:

\begin{prop} \label{thm:bwasy}
Let $(\nu^N)_{N \ge 1}$ be a sequence of probability measures on $\chi_U$ 
such that 
\begin{equation} \label{eq:4.ini}
\lim_{N \to \infty} \nu^N[|\int^{\infty}_{-\infty} g(v) \pi_0^N(dv)
- \int^{\infty}_{-\infty} g(v)\rho_0(v)dv|>\de]=0
\end{equation}
holds for every $\de>0$, $g \in C_0(\R)$ and some function $\rho_0 \in
Y_U$. Then, for every $t>0$,
\begin{equation} \label{eq:4.2-a}
\lim_{N \to \infty} \bar{\mathbb{P}}_{\nu^N}^N[| \int^{\infty}_{-\infty}
g(v) \pi_t^N(dv)- \int^{\infty}_{-\infty} g(v)\rho(t,v)dv|>\de]=0
\end{equation}
holds for every $\de>0$ and $g \in C_0(\R)$, where $\rho(t,v)$ is the unique
classical solution of the following partial differential equation:
\begin{equation}\label{eq:bwasy}
\left\{
\begin{aligned}
\partial_t\rho & = \partial_v^2 \rho + \a \partial_v\big(\rho(1-\rho)\big),\\
\rho(0,\cdot) & = \rho_0(\cdot).
\end{aligned}
\right.
\end{equation}
\end{prop}

Kipnis et al.\ \cite{KOV} also studied the hydrodynamic limit of 
weakly asymmetric simple exclusion processes under the periodic boundary
conditions.

\begin{rem}  \label{rem:4.1}
The unique solution of (\ref{eq:bwasy}) satisfies that
$\rho(t, \cdot) \in Y_U$ for all $t>0$ if $\rho_0\in Y_U$.
This fact (except the equality of two integrals in the definition of $Y_U$)
is seen by regarding the non-linear PDE \eqref{eq:bwasy} as a linear PDE:
$\partial_t\rho = \partial_v^2 \rho + b(t,v) \partial_v\rho$ with
$b(t,v) = \a (1-2\rho(t,v))$, in which $\rho(t,v)$ is considered to be
already given, and then by relying, for instance, on a probabilistic
representation of $\rho(t,v)$: $\rho(t,v) = E_v[\rho_0(X_t^{(t)})]$
in terms of the solution $(X_s) = (X_s^{(t)})_{0\le s \le t}$ of the
stochastic differential equation: $dX_s = \sqrt{2}dB_s + b(t-s,X_s)ds, 
0\le s \le t, X_0 = v$ for each $t>0$, where $B_s$ is the one-dimensional
Brownian motion.  The equality of two integrals: $\int_{-\infty}^0
(1-\rho(t,v))dv = \int_0^\infty \rho(t,v)dv$ follows directly from the
PDE \eqref{eq:bwasy} or by taking limits from the microscopic systems.
\end{rem}

Proposition \ref{thm:bwasy} is formulated only for the test functions
$g$ having compact supports. We also need the following asymptotic
behaviors of the tails of $\pi_t^N$.

\begin{lem} \label{lem:4.3}
Assume that the following two conditions \eqref{eq:4.1a-1} and
\eqref{eq:4.1a-2} hold for $t=0$.  Then, for every $t>0$, we have that
\begin{equation}  \label{eq:4.1a-1}
\lim_{N \to \infty} \bar{\mathbb{P}}_{\nu^N}^N[| \pi_t^N([0,\infty))- 
\int^{\infty}_0\rho(t,v)dv|>\de]=0,
\end{equation}
and 
\begin{equation}  \label{eq:4.1a-2}
\lim_{N \to \infty} \bar{\mathbb{P}}_{\nu^N}^N[| \hat{\pi}_t^N((-\infty,0])-
\int_{-\infty}^0(1-\rho(t,v))dv|>\de]=0,
\end{equation}
for every $\de>0$, where
$$
\hat{\pi}_t^N(dv) = \frac1N \sum_{x\in \Z}(1-\bar{\eta}_t^N(x)) 
\de_{x/N}(dv), \quad t\ge 0, \; v \in \R.
$$
\end{lem}

\begin{proof}
We easily see that \eqref{eq:4.2-a} 
holds for a step function $g=1_{[a,b]}$ with $-\infty<a<b<\infty$,
by approximating such $g$ by a sequence of continuous functions
$g_n \in C_0(\R)$ noting that $0\le \bar{\eta}_t(x), \rho(t,v)\le 1$.
Moreover, Remark \ref{rem:4.1} implies that both 
$\int_{-\infty}^{-K}(1-\rho(t,v))dv$ and
$\int_K^\infty \rho(t,v)dv$
are arbitrarily small for large enough $K>0$.
Thus, to prove \eqref{eq:4.1a-1} and
\eqref{eq:4.1a-2}, it is sufficient to show that for every $\de>0$
there exists $K>0$ such that
\begin{equation}  \label{eq:4.1a-3}
\lim_{N \to \infty} \bar{\mathbb{P}}_{\nu^N}^N[\pi_t^N([K,\infty))
>\de]=0,
\end{equation}
and 
\begin{equation}  \label{eq:4.1a-4}
\lim_{N \to \infty} \bar{\mathbb{P}}_{\nu^N}^N[\hat{\pi}_t^N((-\infty,-K])
>\de]=0,
\end{equation}
respectively.  We prove \eqref{eq:4.1a-3} only, since the proof of
\eqref{eq:4.1a-4} is similar.  To this end, take a function 
$\fa_1\in C_b^2(\R)$ satisfying that $\fa_1' \ge 0$, $\fa_1(v)=1$
for $v\ge 1$ and $\fa_1(v)=0$ for $v\le 1/2$, and set $\fa_K(v):= \fa_1(v/K)$
for $K>0$.  Then,
$$
m_t^N(\fa_K) := \lan \pi_t^N,\fa_K\ran - \lan \pi_0^N,\fa_K\ran 
 - \int_0^t N^2 \bar{L}_{\e(N),U} \lan \pi_s^N,\fa_K\ran ds
$$
is a martingale and the following two bounds:
\begin{align} \label{eq:4.11-b}
& N^2 \bar{L}_{\e(N),U} \lan \pi^N,\fa_K\ran \le 
\|\fa_K''\|_\infty\times |\text{supp }\fa_K''| \le \|\fa_1''\|_\infty/2K,\\
\intertext{and}
& E[m_t^N(\fa_K)^2] \le t \, \|\fa_K'\|_\infty^2
\times |\text{supp }\fa_K'|/N \le t \, \|\fa_1'\|_\infty^2/2KN,
  \label{eq:4.13-b}
\end{align}
hold, where $\lan \pi,\fa\ran = \int_\R \fa(v) \pi(dv)$ and 
$|\text{supp }\fa|$ stands for the Lebesgue measure of the support of
$\fa$.  Indeed a similar computation is made in the proof 
of Proposition \ref{prop:x} below.  
Actually, because of the difference of the generators,
the first sums in \eqref{eq:5.7-a} and \eqref{eq:5.m^N} below should be
taken over $x\in\Z$ rather than $x\in\N$ and the second terms 
do not appear in the present setting.
Moreover, since $\fa_K'\ge 0$, the first sum in \eqref{eq:5.7-a} is bounded
from above by the same sum taken $\e=1$ (because $\e<1$).  However, since
$c_+(x,\bar{\eta}) - c_-(x,\bar{\eta}) = \bar{\eta}(x)-\bar{\eta}(x+1)$,
the bound \eqref{eq:4.11-b} follows by the summation by parts.
Accordingly, we have
\begin{align*}
\pi_t^N([K,\infty)) \le \lan \pi_t^N,\fa_K\ran 
 \le \lan \pi_0^N,\fa_K\ran + t \, \|\fa_1''\|_\infty/2K  + |m_t^N(\fa_K)|.
\end{align*}
Therefore, the condition \eqref{eq:4.1a-1} for $t=0$ controls the behavior
of $\lan \pi_0^N,\fa_K\ran$ and proves \eqref{eq:4.1a-3} with the help of
\eqref{eq:4.13-b}.
\end{proof}

\begin{rem}  \label{rem:4.2}
{\rm (1)} The condition \eqref{eq:4.1a-1} is equivalent to \eqref{eq:4.2-a}
with $g=1_{[0,\infty)}$.  \\
{\rm (2)} The condition \eqref{eq:4.2-a} can be rewritten into an equivalent
form \eqref{eq:4.2-a}$'$, which is \eqref{eq:4.2-a} with $\pi_t^N, \rho(t,v)$
replaced by $\hat{\pi}_t^N, 1-\rho(t,v)$, respectively, and for all
$g\in C_0(\R)$.  Then the condition \eqref{eq:4.1a-2} is equivalent to
\eqref{eq:4.2-a}$'$ with $g=1_{(-\infty,0]}$.
\end{rem}

\subsection{Correspondence between two function spaces $X_U$ and $Y_U$}

We study the relationship between two function spaces $X_U$ and $Y_U$.
To each $\psi\in X_U$, one can associate an element $\rho \in Y_U$ in the
following manner:  First consider a curve $\mathcal{C}_\psi^{(1)}
= \{(u,w); w=\psi(u)\}$  in the first quadrant in the plane, and then 
define a new curve $\mathcal{C}_\psi^{(2)}$ in the upper half plane by 
shifting each point
$(u,w)$ in $\mathcal{C}_\psi^{(1)}$ to $(u-\psi(u),w)$.  The tilt of the 
curve $\mathcal{C}_\psi^{(2)}$ with reversed sign defines the function
$\rho \in Y_U$.  More precisely, for $\psi \in X_U$, we define the function
$G_{\psi}:\R_+^\circ \to \R$ as
\begin{equation}\label{eq:gpsi}
G_{\psi}(u):=u-\psi(u).
\end{equation}
By the definition of $X_U$, $G_{\psi}$ is a monotone function and furthermore
a bijection from $\R_+^\circ$ to $\R$. So, there exists an inverse function of
$G_{\psi}$. We define a function $\Phi_U(\psi):\R \to (0,1)$ as 
$\Phi_U(\psi)(v)=\frac{-\psi^{\prime}(G_{\psi}^{-1}(v))}{1-\psi^{\prime}
(G_{\psi}^{-1}(v))}$ for $v\in \R$. Then, we can easily see that 
$\Phi_U(\psi) \in Y_U$.  In fact, we can show the following proposition.

\begin{prop}
The map $\Phi_U$ defines a one-to-one correspondence between $X_U$ and $Y_U$.  
\end{prop}
\begin{proof}
The inverse map $\Psi_U$ of $\Phi_U$ can be constructed as follows.
For $\rho \in Y_U$, we define two functions $\zeta^-_{\rho}:\R \to \R_+^\circ$ 
and $\zeta^+_{\rho}:\R \to \R_+^\circ$ as
\begin{align} \label{eq:4.2}
\zeta^-_{\rho}(v):=\int^v_{-\infty}(1-\rho(v'))dv' 
\quad \text{ and } \quad
\zeta^+_{\rho}(v):=\int^{\infty}_v \rho(v')dv', \quad v\in\R.
\end{align}
Note that these functions are macroscopic correspondences to those determined
by \eqref{eq:4.1-a}.  By the definition of $Y_U$, $\zeta^-_{\rho}$ and 
$\zeta^+_{\rho}$ are continuously
differentiable monotone functions. Moreover, they are bijections from 
$\R$ to $\R_+^\circ$. So, there exists an inverse function of
 $\zeta^-_{\rho}$. We define a function $\Psi_U(\rho):\R_+^\circ \to 
 \R_+^\circ$ as $\Psi_U(\rho)(u)=
\zeta^+_{\rho}\big((\zeta^-_{\rho})^{-1}(u)\big)$ for
$u \in \R_+^\circ$. Then, we can easily see that $\Psi_U(\rho) \in X_U$. 
Furthermore, $\Psi_U \circ \Phi_U = id_{X_U}$ and $\Phi_U \circ \Psi_U 
= id_{Y_U}$ hold, which concludes the proof.
\end{proof}

\subsection{Proof of Theorem \ref{thm:bthm}}

\noindent
{\it Step 1}.  
We will show that Theorem \ref{thm:bthm} for the process $p_t (\equiv p_t^N)$
follows from
Proposition \ref{thm:bwasy} for the process $\bar{\eta}_t^N$.  To this end, 
we first see that the conditions \eqref{eq:4.ini}, \eqref{eq:4.1a-1}
and \eqref{eq:4.1a-2} at $t=0$ are reduced from the condition \eqref{eq:2.ini}
if we define $\bar{\eta}$ and $\rho_0$ by $\bar{\eta}=\bar{\eta}^p$ and 
$\rho_0=\Phi_U(\psi_0)$, respectively.  

Take $g\in C_b^1(\R)$ satisfying $g(v)=0$ for $v\le -K$ and $g(v)=c$
for $v\ge K$ with some $K>0$ and $c\in\R$.  We will show the condition
\eqref{eq:4.ini} for such $g$; recall Remark \ref{rem:4.2}-(1) for
$t=0$.  For a given $0<\de<1$, determine $u_0, u_1>0$ in such a manner that
$u_0 = \psi_0^{-1}(K+2)\wedge 1$ and $u_1 = 
\psi_0^{-1}(\de)$, respectively.  Now we assume the condition
\begin{equation} \label{eq:4.b-1}
\sup_{u\in[u_0,u_1]} |\tilde{\psi}_p^N(u)-\psi_0(u)| \le \de,
\end{equation}
for $\tilde{\psi}_p^N$.  Then, under this condition, we have that
\begin{equation} \label{eq:4.b-2}
\tilde{\psi}_p^N(u), \psi_0(u) \ge K+1, \quad u \in (0,u_0],
\end{equation}
since $0<\de<1$ and both functions are non-increasing in $u$, and
\begin{equation} \label{eq:4.b-3}
\sharp\{i; \frac{p_i}N > u_0\} = N \tilde{\psi}_p^N(u_0) 
\le N(\psi_0(u_0)+1).
\end{equation}
Thus, under \eqref{eq:4.b-1}, we have that
\begin{align}  \label{eq:4.b-4}
\int_{-\infty}^\infty g(v) \pi_0^N(dv)
& = \frac1N \sum_{x\in\Z} \bar{\eta}(x) g\left( \frac{x}N \right) 
= \frac1N \sum_{i\in\N} g\left( \frac{p_i-i+1}N\right)   \\
& = \frac1N \sum_{i\in\N} g\left( \frac{p_i}N - 
  \tilde{\psi}_p^N(\frac{p_i}N) -\frac{d_i(p)}N  \right) \notag  \\
& = \frac1N \sum_{i\in\N:\frac{p_i}N > u_0} g\left( \frac{p_i}N - 
  \tilde{\psi}_p^N(\frac{p_i}N) -\frac{d_i(p)}N \right) \notag  \\
& = \frac1N \sum_{i\in\N:\frac{p_i}N > u_0} g\left( \frac{p_i}N - 
  \psi_0(\frac{p_i}N) \right) + R^{N,\de,1},
  \notag
\end{align}
where $d_i(p) := \sharp\{j\le i-1; p_j=p_i\}$ is the discrepancy in the
graph of Young diagram $\psi_p(u)$ at $u=p_i$, and 
the error term $R^{N,\de,1}$ satisfies that $|R^{N,\de,1}|\le C_1\de$
with $C_1>0$.  Indeed, the second equality in \eqref{eq:4.b-4} follows
from the fact that $\{x\in\Z;\bar{\eta}(x)=1\} = \{p_i-i+1; i\in\N\}$,
the third from  $\tilde{\psi}^N_p(p_i/N) = \psi_p(p_i)/N
= (i-1-d_i(p))/N$ and the fourth from
\eqref{eq:4.b-2} since $p_i/N\le u_0$ implies that $p_i/N -
\tilde{\psi}_p^N(p_i/N) \le u_0 - (K+1)\le -K$.  The term $R^{N,\de,1}$
in the last line is defined by
$$
R^{N,\de,1} = \frac1N \sum_{i\in\N:\frac{p_i}N > u_0} 
\left\{ g\left( \frac{p_i}N - \tilde{\psi}_N^p(\frac{p_i}N) 
- \frac{d_i(p)}N  \right)
- g\left( \frac{p_i}N - \psi_0(\frac{p_i}N) \right) \right\}
$$
and admits the bound:
\begin{align*}
|R^{N,\de,1}| & \le \frac{\| g'\|_\infty}N \sum_{i\in\N:\frac{p_i}N > u_0} 
\left\{ \left| \tilde{\psi}_N^p\left(\frac{p_i}N \right) 
  - \psi_0\left(\frac{p_i}N \right) \right| + \frac{d_i(p)}N  \right\} \\
& \le \| g'\|_\infty\cdot 4\de(\psi_0(u_0)+1),
\end{align*}
since the first summand in the above sum is bounded by
$\de$ if $u_0\le p_i/N\le u_1$ under the condition \eqref{eq:4.b-1} and
is bounded by $2\de$ if $p_i/N\ge u_1$ by noting that 
$0\le \tilde{\psi}_N^p(u), \psi_0(u)\le 2\de$ for $u\ge u_1$
which follows from the monotonicity of these functions,
and its second summand is bounded by $\tilde{\psi}_p^N(p_i/N-)
-\tilde{\psi}_p^N(p_i/N)$ which is further bounded by $2\de$ from 
\eqref{eq:4.b-1} recalling the continuity of $\psi_0$; we have also
used \eqref{eq:4.b-3}.  We can further rewrite the sum in the last term
of \eqref{eq:4.b-4} as
\begin{align*}
\frac1N \sum_{i\in\N:\frac{p_i}N > u_0} & g\left( \frac{p_i}N - 
  \psi_0(\frac{p_i}N) \right) 
= \frac1N \sum_{i \in \N}(g \circ G_{\psi_0})\left(\frac{p_i}{N}\right) \\
& = \frac1N \sum_{i \in \N} \int^{\frac{p_i}{N}}_0 
  (g \circ G_{\psi_0})^{\prime}(u)du
= \int^{\infty}_0 (g \circ G_{\psi_0})^{\prime}(u) \tilde{\psi}_{p}^N(u)du.
\end{align*}
Note that, since $g \circ G_{\psi_0}(u) = g(u-\psi_0(u))=0$ if 
$u\in (0,u_0]$, we have dropped the condition $p_i/N> u_0$ from the summand
of the above sums, and, by the same reason, we can replace the region of
the integral in the last line from $[0,\infty)$ to $[u_0,\infty)$.
Consider the error $R^{N,\de,2}$ defined by
$$
R^{N,\de,2} = \int^{\infty}_0 (g \circ G_{\psi_0})^{\prime}(u) 
\left\{\tilde{\psi}_{p}^N(u) -\psi_0(u)\right\} du,
$$
which can be bounded as
$$
|R^{N,\de,2}| \le 
2\de \int^{\tilde{K}}_{u_0} |(g \circ G_{\psi_0})^{\prime}(u)| du
= C_2\de.
$$
where $\tilde{K}$ is determined in such a manner that $(g \circ G_{\psi_0})'
(v)=0$ for $v \ge \tilde{K}$. Furthermore, by the integration by parts 
formula, we have that
\begin{align*}
\int^{\infty}_0 (g \circ G_{\psi_0})^{\prime}(u)\psi_0(u)du
& =-\int^{\infty}_0 (g \circ G_{\psi_0})(u)\psi_0^{\prime}(u)du \\
& = - \int^{\infty}_{-\infty} g(v)\psi_0^{\prime}(G_{\psi_0}^{-1}(v))
\frac{1}{1-\psi_0^{\prime}(G_{\psi_0}^{-1}(v))}dv \\
&=\int^{\infty}_{-\infty} g(v)\Phi_U(\psi_0)(v)dv=\int^{\infty}_{-\infty}
g(v)\rho_0(v)dv.
\end{align*}
Therefore, under the condition \eqref{eq:4.b-1}, we have shown that
\begin{equation*}
\left|\int_{-\infty}^\infty g(v) \pi_0^N(dv)
- \int_{-\infty}^\infty g(v) \rho_0(v)dv \right| \le (C_1+C_2)\de.
\end{equation*}
This implies the condition \eqref{eq:4.ini} for $\pi_0^N$ and
$g\in C_b^1(\R)$ satisfying $g(v)=0$ for $v\le -K$ and $g(v)=c$
for $v\ge K$ with some $K>0$ and $c\in\R$.

The same condition \eqref{eq:4.ini} with $\pi_0^N, \rho_0$ replaced by
$\hat{\pi}_0^N, 1-\rho_0$, respectively, and
$g\in C_b^1(\R)$ satisfying $g(v)=0$ for $v\ge K$ and $g(v)=c$
for $v\le -K$ with some $K>0$ and $c\in\R$ can be shown by symmetry; recall
Remark \ref{rem:4.2}-(2) for $t=0$.  Indeed, for each $p \in \mathcal{P}$, 
we denote by $\check{p} = (\check{p}_i)_{i\in\N}$ the mirror image of the
Young diagram $p$ with the axis of symmetry $\{y=x\}$ in the plane,
i.e.\ $\check{p}_i=\sharp\{j;p_j \ge i\}$.  Similarly, we denote by 
$\check{\psi_0}$ the mirror image of the curve $\psi_0$ with the axis of 
symmetry $\{y=x\}$, i.e.\ $\check{\psi_0}(u):=\psi_0^{-1}(u)$.
Then, the condition \eqref{eq:2.ini} with $\tilde{\psi}_{p}^N$, $\psi_0$ 
replaced by $\tilde{\psi}_{\check{p}}^N$, $\check{\psi}_0$, respectively,
is reduced from \eqref{eq:2.ini} itself. Therefore, if we denote by $\check{\pi}_0^N$ the scaled empirical measure of the configuration $\bar{\eta}^{\check{p}}$ and $\check{\rho}_0$ the function associated with $\check{\psi}_0$ by the one-to-one map constructed in Subsection 4.2, namely
$\check{\pi}_0^N(dv) = \frac1N \sum_{x\in \Z}\bar{\eta}^{\check{p}}(x)\de_{\frac{x}{N}}(dv)$ and $\check{\rho}_0=\Phi_U(\check{\psi}_0)$, then 
we see that the condition \eqref{eq:4.ini} with $\pi_0^N, \rho_0$
replaced by $\check{\pi}_0^N, \check{\rho}_0$, respectively,
holds for every $\de>0$ and $g\in C_b^1(\R)$ satisfying $g(v)=0$ for 
$v\le -K$ and $g(v)=c$ for $v\ge K$ with some $K>0$ and $c\in\R$ by the 
above mentioned argument. However, since we easily see the relations: 
$\bar{\eta}^{\check{p}}(x)=1-\bar{\eta}^p(-x)$ and 
$\check{\rho}_0(u)=1-\rho_0(-u)$, the condition \eqref{eq:4.ini}
with $\pi_0^N, \rho_0$ replaced by $\hat{\pi}_0^N, 1-\rho_0$, respectively,
is shown for $g\in C_b^1(\R)$ satisfying $g(v)=0$ for 
$v\ge K$ and $g(v)=c$ for $v\le -K$.

\noindent
{\it Step 2}.  
In order to complete the proof of the theorem, it is now sufficient
to show that \eqref{eq:4.2-a} in Proposition \ref{thm:bwasy} together with
the assertions in
Lemma \ref{lem:4.3} implies \eqref{eq:2.t} with $\psi_t =\Phi_U(\rho_t)$.
The non-linear equation \eqref{eq:bhydro} for $\psi_t$ follows from
\eqref{eq:bwasy} for $\rho_t$. 

Take $f\in C_0(\R_+^\circ)$ and $t>0$ arbitrarily and fix them throughout 
the rest of the proof.  Then we have that
\begin{equation} \label{eq:4.a-0}
\int^{\infty}_0 f(u) \tilde{\psi}_{p_t}^N(u)du
= \frac1N \sum_{x\in\Z} F\left(\frac{\zeta^-_t(x)}N\right)\bar{\eta}_t(x),
\end{equation}
where $F(u) = \int_0^u f(u')du'$ and $\zeta^-_t(x) = 
\zeta^-_{\bar{\eta}_t}(x)$
defined by \eqref{eq:4.1-a}.  For a given $\de>0$, take $K>0$ such that
\begin{equation} \label{eq:4.a-3}
\int_{-\infty}^{-K}(1-\rho(t,v))dv < \de/6,
\qquad
\int^{\infty}_K \rho(t,v)dv < \de/6,
\end{equation}
and the conditions \eqref{eq:4.1a-3} and \eqref{eq:4.1a-4} hold with
$\de$ replaced by $\de/3$, recall the proof of Lemma \ref{lem:4.3}.

Now let us prove that
\begin{equation} \label{eq:4.b}
\lim_{N \to \infty} \mathbb{P}_{\nu^N}^N[
\sup_{x\in \Z: |x/N-v|\le \th} |\frac{\zeta^-_t(x)}N - \int_{-\infty}^v
(1-\rho(t,v'))dv'| >\de] =0
\end{equation}
holds for every $0<\th<\de/3$ and $v\in \mathcal{V}_{K,\th} := \{v\in\R;
|v|\le K+1, v\in \th\Z\}$.  In fact, since $\zeta^-_t(x)$ is 
non-decreasing in $x$, we have that
\begin{equation} \label{eq:4.c}
\hat{\pi}_t^N((-\infty,v-\th]) \le \frac{\zeta^-_t(x)}N 
= \hat{\pi}_t^N((-\infty,x/N]) \le \hat{\pi}_t^N((-\infty,v+\th])
\end{equation}
for $x\in \Z$ such that $|x/N-v|\le \th$.  However, from \eqref{eq:4.1a-4}
and \eqref{eq:4.2-a} with $g= 1_{[-K,v\pm \th]}$ and $\de$ replaced by 
$\de/3$, we have that
\begin{equation} \label{eq:4.d}
\lim_{N \to \infty} \mathbb{P}_{\nu^N}^N[
|\hat{\pi}_t^N((-\infty,v\pm\th]) - \int_{-\infty}^{v\pm\th}
(1-\rho(t,v'))dv'| >2\de/3] =0.
\end{equation}
Moreover, since 
$
|\int_{-\infty}^{v\pm\th}(1-\rho(t,v'))dv' - \int_{-\infty}^{v}
(1-\rho(t,v'))dv'| \le \th,
$
if $0<\th<\de/3$, \eqref{eq:4.c} and \eqref{eq:4.d} imply \eqref{eq:4.b}.
Since $\|F'\|_\infty=\|f\|_\infty<\infty$, \eqref{eq:4.b} further shows that
\begin{equation} \label{eq:4.e}
\lim_{N \to \infty} \mathbb{P}_{\nu^N}^N[
\sup_{x\in \Z: |x/N-v|\le \th} \left|F\left(\frac{\zeta^-_t(x)}N\right) - 
F\left(\int_{-\infty}^v (1-\rho(t,v'))dv'\right)\right| >\de\|f\|_\infty] =0
\end{equation}
for every $v\in \mathcal{V}_{K,\th}$ if $0<\th<\de/3$.

We now return to the formula \eqref{eq:4.a-0} and divide it as
\begin{align*}
\int^{\infty}_0 f(u) \tilde{\psi}_{p_t}^N(u)du
 =: I_1^N + I_2^N + I_3^N,
\end{align*}
where $I_1^N$, $I_2^N$ and $I_3^N$ are defined as the sums in the right
hand side of \eqref{eq:4.a-0} restricted for $x\le -KN$, $-KN<x<KN$
and $x\ge KN$, respectively.
For the first term $I_1^N$, since $f\in C_0(\R_+^\circ)$, we see that
$f(u)=0$ so that $F(u)=0$ for $u\in[0,u_0]$ with some $u_0>0$.
Therefore, choosing $\de>0$ such that $\de/3<u_0$, \eqref{eq:4.1a-4}
with $\de$ replaced by $\de/3$ implies that
\begin{equation} \label{eq:4.f}
\lim_{N \to \infty} \mathbb{P}_{\nu^N}^N[ I_1^N =0] =1.
\end{equation}
For the second term $I_2^N$, by \eqref{eq:4.e}, we can show that
\begin{equation*}
\lim_{N \to \infty} \mathbb{P}_{\nu^N}^N[ |I_2^N - \tilde{I}_2^N|>\de] =0,
\end{equation*}
where, assuming $K/\th\in \Z$ for simplicity,
$$
\tilde{I}_2^N = \frac1N \sum_{k=-K/\th}^{K/\th-1}
F\big(\int_{-\infty}^{k\th} (1-\rho(t,v'))dv'\big)
\sum_{k\th \le x/N<(k+1)\th} \bar{\eta}_t(x).
$$
However, by applying \eqref{eq:4.2-a} with $g=1_{[k\th,(k+1)\th)}$ again,
we have that
\begin{equation*}
\lim_{N \to \infty} \mathbb{P}_{\nu^N}^N[ |\tilde{I}_2^N 
- \bar{I}_2^\th|>\de] =0,
\end{equation*}
where
$$
\bar{I}_2^\th = \sum_{k=-K/\th}^{K/\th-1}
F\big(\int_{-\infty}^{k\th} (1-\rho(t,v'))dv'\big)
\int_{k\th}^{(k+1)\th} \rho(t,v')dv'.
$$
By letting $\th\downarrow 0$, $\bar{I}_2^\th$ converges to
$$
I_K = \int_{-K}^K F\big(\int_{-\infty}^v (1-\rho(t,v'))dv'\big)
 \rho(t,v)dv.
$$
For the third term $I_3^N$, since $0\le I_3^N \le \|F\|_\infty 
\pi_t^N([K,\infty))$, we see from \eqref{eq:4.1a-3}
with $\de$ replaced by $\de/3$ that
\begin{equation*}
\lim_{N \to \infty} \mathbb{P}_{\nu^N}^N[ I_3^N > \de\|F\|_\infty/3
] =0.
\end{equation*}

These computations are now summarized into
\begin{equation*}
\lim_{N \to \infty} \mathbb{P}_{\nu^N}^N[|\int^{\infty}_0 
f(u) \tilde{\psi}_{p_t}^N(u)du - I|>\de]=0,
\end{equation*}
where
$$
I = \int_{-\infty}^\infty F\big(\int_{-\infty}^v (1-\rho(t,v'))dv'\big)
 \rho(t,v)dv.
$$
Note that $I_K$ coincides with $\int_{-\infty}^K F\big(\int_{-\infty}^v 
(1-\rho(t,v'))dv'\big)  \rho(t,v)dv$ because of \eqref{eq:4.a-3} recalling
that $\de/3<u_0$ and the integration over $[K,\infty)$ in $v$ can be taken
small enough if $K$ is sufficiently large.  However, by the change of 
variables $w= \zeta^-_{\rho_t}(v)$ and the integration by parts, we have
that
\begin{align*}
I & = \int_{-\infty}^\infty F\big(\zeta^-_{\rho_t}(v)\big) \rho(t,v)dv 
= -\int_{-\infty}^\infty \int_0^{\zeta^-_{\rho_t}(v)}f(u)du \cdot
\frac{d\zeta^+_{\rho_t}}{dv}(v)dv \\
&= -\int_0^\infty \int_0^wf(u)du \cdot
\frac{d\zeta^+_{\rho_t}}{dv}\big( (\zeta^-_{\rho_t})^{-1}(w)\big)
\frac{dv}{dw}dw  \\
& = -\int_0^\infty \int_0^wf(u)du \cdot
\frac{d}{dw}\left( \zeta^+_{\rho_t} \big( (\zeta^-_{\rho_t})^{-1}
 (w)\big)\right)dw  \\
& = \int_0^\infty f(u) \zeta^+_{\rho_t}\big( (\zeta^-_{\rho_t})^{-1}
(u)\big) du = \int_0^\infty f(u) \Psi_U(\rho_t)(u) du.
\end{align*}
This completes the proof of Theorem \ref{thm:bthm}.

\section{Proof of Theorem \ref{thm:fthm}}

This section gives the proof of Theorem \ref{thm:fthm} for the RU-case,
i.e.\ the case corresponding to the restricted uniform statistics.
Similarly to the process $\xi_t$ in the U-case, we consider the particle
numbers (or the gradient of the height function $\psi_{q_t}$) 
$\eta_t = (\eta_t(x))_{x\in \Z_+}$  defined by $\eta_t(x) =
\sharp\{i; q_i(t)=x\} \in \{0,1\}$ for $x\in \N$ and $\eta_t(0)=
\infty$.  Note that only $0$-$1$ height differences are allowed under the
restriction imposed on the Young diagrams $q\in \mathcal{Q}$.
Then $\eta_t$ becomes the weakly asymmetric simple exclusion
process with the stochastic reservoir at $\{0\}$, which provides particles 
into the region $\N$ with rate $\e$ and absorbs them with rate $1$. 
Contrarily to the weakly asymmetric simple exclusion process $\bar{\eta}_t$
on $\Z$ considered in the U-case, $\eta_t$ determines a finite particles'
system on $\N$.

In the RU-case, one does not have a nice transformation for $\eta_t$,
which removes the stochastic reservoir as in the U-case.  We will apply
the Hopf-Cole transformation for $\eta_t$ at the microscopic level, which
linearizes the leading term, and study the boundary behavior of the
transformed process.

\subsection{The process $\eta_t$}

Denote by $\chi_R$ the state space of the process $\eta_t$ defined from
$q_t$:
\begin{displaymath}
\chi_R:=\{\eta \in \{0,1\}^{\N}; \sum_{x \in \N}\eta(x) \ < \ \infty\}.
\end{displaymath}
We have a one-to-one correspondence between $\chi_R$ and $\mathcal{Q}$. 
Indeed, for $\eta \in \chi_R$, we assign $q^{\eta} \in \mathcal{Q}$ by the
following rule:
\begin{displaymath}
q^{\eta}_i= \min \{x \in \Z_+; \sum_{y \ge x+1}\eta(y) \le i-1\},
 \quad i \in \N.
\end{displaymath}
In other words, $\{q_i^\eta\}_{i\in\N}$ is determined by numbering the
set $\{x\in\N; \eta(x)=1\}$ from the right and, if $i$ is larger than
the cardinality of this set, we define $q_i^\eta=0$.
We can show that the map $\eta \to q^{\eta}$ is well-defined and also it 
is a bijection from $\chi_R$ to $\mathcal{Q}$. So we denote its inverse map
by $q \to \eta^q$. 

We now consider the Markov process $\eta_t$ on $\chi_R$ with the generator $\bar{L}_{\e,R}$ acting on functions $f:\chi_R \to \R$ as
$$
\bar{L}_{\e,R}f(\eta)= \bar{L}_{\e,R}^i f(\eta) +\bar{L}_{\e,R}^bf(\eta),
$$
where
\begin{align*}
\bar{L}_{\e,R}^if(\eta) &=\sum_{x \in \N} 
\big[\e c_+(x,\eta) + c_-(x,\eta) \big]\{f(\eta^{x,x+1})-f(\eta)\}  \\
\intertext{and}
\bar{L}_{\e,R}^bf(\eta)&= \big[\e 1_{\{\eta(1)=0\}}
+1_{\{\eta(1)=1\}} \big]\{f(\eta^{1})-f(\eta)\}
\end{align*}
are the interior and boundary terms of the generator, respectively,
$c_+(x,\eta), c_-(x,\eta)$ and $\eta^{x,y}$ are defined by 
\eqref{eq:4.c-jump}, \eqref{eq:4.eta} with $\bar{\eta}$ replaced by $\eta$,
respectively, and
\begin{displaymath}
 \eta^{1}(z) = \begin{cases}
	\eta(z)  & \text{if $z \neq 1$}, \\
	1-\eta(1)  & \text{if $z=1$}. 
\end{cases} 
\end{displaymath}
The following lemma is easy so that the proof is omitted.

\begin{lem}
Two processes $\{q_t\}_{t\ge 0}$ and $\{q^{\eta_t}\}_{t\ge 0}$ have the same
distributions on $D(\R_+,\mathcal{Q})$.
\end{lem}

For a probability measure $\nu$ on $\chi_R$ and $N\ge 1$, we denote by 
$\bar{\mathbb{Q}}_{\nu}^N$ the distribution on $D(\R_+,\chi_R)$ of the 
process $\eta_t^N$ with generator $N^2 \bar{L}_{\e(N),R}$
and the initial measure $\nu$, where  $\e(N)$ is defined by \eqref{eq:2.eF}.
Let us define the scaled empirical measures
$\pi_t^N(dv), t\ge 0, v\in\R_+^\circ$ of the process $\eta_t^N$ by the formula
\eqref{eq:4.pi} with $\bar{\eta}_t^N$ replaced by $\eta_t^N$ and
the sum taken over all $x\in \N$ rather than $x\in\Z$.

The hydrodynamic limit for a boundary driven exclusion process is studied
by \cite{ELS-2}.  Our model involves
a weak asymmetry both in dynamics and the boundary condition, and 
furthermore it is defined on an infinite volume $\N$.  
Note that the boundary generator $\bar{L}_{\e,R}^b$ is invariant
under the Bernoulli measure with mean $\rho^\e = \e/(1+\e)$.  This actually
determines the Dirichlet boundary condition at $v=0$ in the limit equation
\eqref{eq:fwasy} stated below, since $\rho^\e$ converges to $1/2$ as 
$\e=\e(N)\uparrow 1$.  The hydrodynamic limit for models in infinite volume
was discussed by several authors including \cite{LM}.  It might be possible
to apply these methods to our model, but we will employ the simplest
way based on the Hopf-Cole transformation.

\subsection{Hopf-Cole transformation}

In this subsection we introduce the microscopic Hopf-Cole transformation
for the process $\eta_t^N$ and formulate Theorem \ref{thm:hopf} on its
hydrodynamic behavior.  Theorem \ref{thm:fthm} will be shown from
Theorem \ref{thm:hopf} in Subsection 5.4.

It is well-known that the (macroscopic) Hopf-Cole transformation:
\begin{displaymath}
\omega(t,u) = \exp \{ \b \int_{u}^{\infty} \rho_t(v)dv \},  
\quad u\in \R_+
\end{displaymath}
allows us to reduce the solution of the viscous Burgers' equation 
\eqref{eq:fwasy} (at least on the whole line $\R$) to that of
the linear diffusion equation \eqref{eq:omega} (on $\R$).
We introduce the corresponding transformation at the microscopic level,
cf.\ \cite{G}.  Namely, we consider the process $\zeta_t^N = 
(\zeta_t^N(x))_{x\in\N}$ defined by $\zeta_t^N(x)
:=\exp \big\{- (\log \e) \sum_{y=x}^{\infty}\eta_t^N(y) \big\}$ 
with $\e=\e(N)$ from the process $\eta_t^N$ and the $C(\R_+)$-valued
process $\tilde{\zeta}^N(t,u), u\in \R_+$ by interpolating
$\tilde{\zeta}^N(t,u):=\zeta_t^N(Nu)$ defined for $u \in \N/N$ 
in such a manner that 
$$
\tilde{\zeta}^N(t,u):=\exp \left[- (\log \e) 
\left\{\sum_{y=[Nu]+1}^{\infty}\eta_t^N(y)+
1_{\{u\ge 1/N\}} ([Nu]+1-Nu)\eta_t^N([Nu])\right\} \right], 
$$ 
for $u\in \R_+$.

\begin{thm}\label{thm:hopf}
Let $(\nu^N)_{N \ge 1}$ be a sequence of probability measures on $\chi_R$
such that 
\begin{equation}\label{eq:5.pi}
\lim_{N \to \infty} \nu^N[|\int^{\infty}_0 g(v) \pi_0^N(dv)
- \int^{\infty}_0 g(v)\rho_0(v)dv|>\de]=0
\end{equation}
holds for every $\de>0$, $g \in C_b(\R_+)$ satisfying $g(v)=c$
for $v\ge K$ with some $K>0$ and $c\in\R$, and some continuous 
function $\rho_0: \R_+ \to [0,1]$ satisfying $\int_0^\infty \rho_0(v)dv
<\infty$.  Then, for every $T>0$, $K>0$ and $\de>0$,
\begin{equation}\label{eq:5.zeta}
\lim_{N \to \infty} \bar{\mathbb{Q}}_{\nu^N}^N[\sup_{0 \le t \le T, 0 \le u \le K}|\tilde{\zeta}^N(t,u)- \omega(t,u)|>\de]=0
\end{equation}
holds, where $\omega(t,u)$ is the unique bounded weak solution of the 
following linear diffusion equation: 
\begin{equation} \label{eq:omega}
\left\{
\begin{aligned}
\partial_t\omega & = \partial_u^2 \omega + \b \partial_u \omega,  
  \quad u\in \R_+,  \\
\omega(0,u) & = \exp \{ \b \int_{u}^{\infty} \rho_0(v)dv \},  \quad u\in 
\R_+,  \\
2 \partial_u \omega(t,0)& + \b \omega(t,0) =0, \quad t>0,\\
\omega(t,\infty) &=1, \quad t>0.
\end{aligned}
\right.
\end{equation}
Namely, for every $t>0$,
\begin{align}  \label{eq:5.w}
\int^{\infty}_0 g(u) \omega(t,u)du =& \int^{\infty}_0 g(u)\omega(0,u)du \\
& + \int^{t}_0 \int^{\infty}_0 \big(g''(u)-\b g'(u)\big)
 \omega(s,u)du ds   \notag
\end{align}
holds for every $g \in C^2_0(\R_+)$ satisfying
$2 g'(0) - \b g(0)=0$ and 
$\lim_{u \to \infty} \omega(t,u)=1$.
\end{thm}

The following corollary, which gives the hydrodynamic limit for $\eta_t^N$,
is an immediate consequence of Theorem \ref{thm:hopf} and will be used
in \cite{FS-1}.

\begin{cor}\label{cor:fwasy}
Under the same assumption as Theorem \ref{thm:hopf},
\begin{equation*}
\lim_{N \to \infty} \bar{\mathbb{Q}}_{\nu^N}^N[|\int^{\infty}_0
g(v) \pi_t^N(dv)- \int^{\infty}_0 g(v)\rho(t,v)dv|>\de]=0
\end{equation*}
holds for every $t>0, \de>0$ and $g \in C_0(\R_+^\circ)$, where $\rho(t,u)$ is
the unique classical solution of the following partial differential equation:
\begin{equation}\label{eq:fwasy}
\left\{
\begin{aligned}
\partial_t\rho & = \partial_v^2 \rho + \b \partial_v(\rho(1-\rho)),  
  \quad v\in \R_+,  \\
\rho(0,v) & = \rho_0(v),  \quad v\in \R_+,  \\
\rho(t,0) &=1/2, \quad t>0.
\end{aligned}
\right.
\end{equation}
\end{cor}

\subsection{Proof of Theorem \ref{thm:hopf}}

This subsection proves Theorem \ref{thm:hopf}.

\subsubsection{Uniform estimate on the total mass}

We prepare a proposition which gives a uniform estimate on the scaled total
mass of $\eta_t^N$.  For the proof, the conditions \eqref{eq:5.pi}
with $g\equiv 1$ and  $\int_0^\infty \rho_0(v)dv<\infty$ are essential.

\begin{prop}\label{prop:x}
Denote by $X_t^N$ the process of the total mass of the empirical measure
$\pi_t^N$, namely $X_t^N:=\frac{1}{N}\sum_{x \in \N}\eta_t^N(x)
\big(\equiv \pi_t^N(\R_+^\circ)\big)$.  Then, for every $T>0$, we have that
\begin{equation}\label{eq:5.prop.5.4}
\lim_{\la\to\infty} \sup_{N\ge 1}  \bar{\mathbb{Q}}_{\nu^N}^N
\left[\sup_{0 \le t \le T}X_t^N >\la \right] =0.
\end{equation}
\end{prop}

\begin{proof}
For $\fa \in C_b^2(\R_+^\circ)$, denote by $m_t^N(\fa)$ the martingale 
defined by 
\begin{displaymath}
m_t^N(\fa):=\lan\pi_t^N, \fa\ran-\lan\pi_0^N,\fa\ran 
-\int^t_0 N^2 \bar{L}_{\e(N),R} \lan\pi_s^N, \fa\ran ds.
\end{displaymath}
Then, by a simple computation, we have that
\begin{align} \label{eq:5.7-a}
N^2 \bar{L}_{\e(N),R} \lan\pi^N, \fa\ran
= N \sum_{x\in \N}& \big(\fa((x+1)/N) - \fa(x/N)\big) 
 \big\{ \e c_+(x,\eta) - c_-(x,\eta) \big\} \\
& + N \fa(1/N) \left\{ \e 1_{\{\eta(1)=0\}}
- 1_{\{\eta(1)=1\}} \right\},   \notag
\end{align}
and 
\begin{align} \label{eq:5.m^N}
\frac{d}{dt} \lan m^N(\fa)\ran_t
= \sum_{x\in \N}& \big(\fa((x+1)/N) - \fa(x/N)\big)^2 
\big\{ \e c_+(x,\eta_t^N) + c_-(x,\eta_t^N) \big\} \\
& + \fa(1/N)^2
\left\{ \e 1_{\{\eta_t^N(1)=0\}}
+ 1_{\{\eta_t^N(1)=1\}} \right\},   \notag
\end{align}
for the quadratic variation of $m_t^N(\fa)$, if the right hand sides of
these equalities converge absolutely.
Now take a function $\fa \in C_b^2(\R_+^\circ)$ such that $\fa'\ge 0$,
$\fa(u)=0$ for $0 < u \le 1$ and $\fa(u)=1$ for $u \ge 2$. 
Then, \eqref{eq:5.7-a} shows that
\begin{displaymath}
N^2 \bar{L}_{\e(N),R} \lan\pi^N, \fa\ran \le \|\fa''\|_\infty,
\end{displaymath}
similarly to the proof of \eqref{eq:4.11-b}.  Therefore, 
\begin{align*}
\sup_{0 \le t \le T}\lan\pi_t^N, \fa\ran & 
\le \lan\pi_0^N,\fa\ran+T\|\fa''\|_\infty
+\sup_{0 \le t \le T}|m_t^N(\fa)| \\
	& \le X_0^N +T\|\fa''\|_\infty +1 + \sup_{0 \le t \le T} m_t^N(\fa)^2,
\end{align*}
where we have estimated the martingale as $|m_t^N(\fa)| \le 
m_t^N(\fa)^2+1$.  One can apply Doob's inequality to show
$\lim_{N\to\infty} E[\sup_{0 \le t \le T} m_t^N(\fa)^2]=0$ from
\eqref{eq:5.m^N}, which, in
particular, proves $\sup_N E[\sup_{0 \le t \le T} m_t^N(\fa)^2]<\infty$.
Since the assumption of Theorem \ref{thm:hopf} (especially \eqref{eq:5.pi}
with $g\equiv 1$ and the integrability of $\rho_0$) implies that
$\lim_{\la\to\infty} \sup_N \nu^N(X_0^N>\la) =0$, the conclusion of the
proposition follows by the inequality:
$X_t^N \le 2+\lan\pi_t^N, \fa\ran$.
\end{proof}

\subsubsection{Tightness of $\{\tilde{\zeta}^N\}_N$}

Let $P_N$ be the probability distribution of $\tilde{\zeta}^N = 
\{\tilde{\zeta}^N(t,u)\}$ on $D([0,T],C(\R_+))$, where the space
$C(\R_+)$ is endowed with the topology determined by the uniform
convergence on every compact set of $\R_+$. 

\begin{lem}
The family of probability measures $\{P_N\}_{N\ge1}$ is relatively compact.
\end{lem}

\begin{proof}
To conclude the lemma, by Prokhorov's theorem, it suffices to show the 
following three conditions for $\{P_N\}_{N\ge1}$:
\begin{align*}
&\; \text{(i) For every }  t \in [0,T],   \lim_{\lambda \to \infty}\sup_{N\ge1}
P_N[\tilde{\zeta}(t,0) > \lambda]=0. \\
&\, \text{(ii) For every $\de>0$ and }  t \in [0,T],  
\lim_{\ga \downarrow 0}\sup_{N\ge1} P_N[\sup_{|u-v| \le \ga}
|\tilde{\zeta}(t,u)-\tilde{\zeta}(t,v)| > \de]=0. \\
&\text{(iii) For every $\de>0$ and $K>0$,}  \lim_{\gamma \downarrow 0}
\limsup_{N\to\infty} P_N[\sup_{|t-s| \le \gamma,0 \le u \le K}
 |\tilde{\zeta}(t,u)-\tilde{\zeta}(s,u)|>\de]=0.
\end{align*}

By the relation: $\tilde{\zeta}^N(t,0) = \exp\{ - (\log\e) N X_t^N\}$,
we have that
\begin{align*}
P_N[\tilde{\zeta}^N(t,0) > \lambda] 
\le \bar{\mathbb{Q}}_{\nu^N}^N[X_t^N > \log \lambda/C],
\end{align*}
note that there exists $C>0$ such that $0<-\log\e\le C/N$ for
$\e=\e(N)$ and every $N\ge 1$.  Proposition \ref{prop:x} proves (i).

Since $\tilde{\zeta}^N(t,\cdot)$ is a non-increasing function, 
for every $0\le u < v$, we have that
\begin{align} \label{eq:5.Dzeta-1}
|\tilde{\zeta}^N(t,u)-\tilde{\zeta}^N(t,v)| 
&  \le \tilde{\zeta}^N(t,u)\left[\exp \big\{- (\log \e) I^N(t,u,v) \big\}
  -1\right] \\
& \le \tilde{\zeta}^N(t,0) \left[\exp\{C I^N(t,u,v)/N\}-1 \right],
   \notag
\end{align}
where
\begin{align*}
I^N(t,u,v) :=
\sum_{y=[Nu]+1}^{[Nv]}\eta_t^N(y)+([Nu]+1-Nu)\eta_t^N([Nu])
- ([Nv]+1-Nv)\eta_t^N([Nv]),
\end{align*}
which has a trivial bound: $I^N(t,u,v) \le N(v-u)$.
Therefore, we have that
\begin{align} \label{eq:5.Dzeta-2}
P_N[\sup_{|u-v| \le \ga}|\tilde{\zeta}^N(t,u)-\tilde{\zeta}^N(t,v)| > \de]
 & \le \bar{\mathbb{Q}}_{\nu^N}^N[e^{C X_t^N} (e^{C\ga}-1) > \de] \\
 & = \bar{\mathbb{Q}}_{\nu^N}^N[ X_t^N > \log(\de/(e^{C\ga}-1))/C ].
 \notag
\end{align}
Proposition \ref{prop:x} concludes (ii).

Finally we prove (iii).
By the definition of $\tilde{\zeta}^N(t,u)$ and Proposition \ref{prop:x},
we only need to show that for every $K>0$ and $\de>0$,
\begin{displaymath}
\lim_{\gamma \downarrow 0} \limsup_{N\to\infty} 
\bar{\mathbb{Q}}_{\nu^N}^N
[\sup_{|t-s| \le \gamma, 0 \le u \le K} |\frac{1}{N}\sum_{x=[Nu]}^{\infty}
\eta_t^N(x)-\frac{1}{N}\sum_{x=[Nu]}^{\infty}\eta_s^N(x)|>\de]=0.
\end{displaymath}
Noting that $\frac{1}{N}\sum_{x=[Nu]}^{\infty}\eta_t^N(x)=\lan\pi_t^N, 1_{[u,\infty)}\ran+\frac{1}{N}\eta_t^N([Nu])$, we consider smooth functions $\phi_{\k}(u,\cdot)$ which approximate the function $1_{[u,\infty)}$ as $\k\downarrow 0$
such that
\begin{align*}
\phi_{\k}(u,v)=0 & \quad \text{for $v \le u-\k$} \\
0 \le \phi_{\k}(u,v) \le 1 & \quad \text{for $u-\k \le v \le u+\k$} \\
\phi_{\k}(u,v)=1 & \quad \text{for $v \ge u+\k$} \\
\phi_{\k}(u,\cdot)= \phi_{\k}(u+v, \cdot +v) & \quad \text{for every $u$ and $v$}.
\end{align*}
In particular, we have that
\begin{displaymath}
| \lan\pi_t^N, \phi_{\k}(u,\cdot)\ran-\lan\pi_t^N, 1_{[u,\infty)}\ran |  \le \k \quad \text{for every $u$}. \\
\end{displaymath}
Moreover, $\|\phi_{\k}\|_{2,\infty}
:=\sup_{u} \{\|\phi_{\k}^{\prime}(u,\cdot)\|_{\infty}
 + \|\phi_{\k}^{\prime \prime}(u,\cdot)\|_{\infty}\}$ is finite. 
Now, it is enough to prove that for every $\k, \delta>0$,
\begin{displaymath}
\lim_{\gamma \downarrow 0} \limsup_{N\to\infty} 
\bar{\mathbb{Q}}_{\nu^N}^N[\sup_{|t-s| \le \gamma, 
\k \le u \le K} |\lan\pi_t^N, \phi_{\k}(u,\cdot)\ran-\lan\pi_s^N, \phi_{\k}(u,\cdot)\ran|>\delta]=0.
\end{displaymath}
However, the term $\lan\pi_t^N, \phi_{\k}(u,\cdot)\ran-\lan\pi_s^N, 
\phi_{\k}(u,\cdot)\ran$ is rewritten as 
\begin{displaymath}
\int^t_s \sum_{x \in \N}\phi_{\k}(u,\frac{x}{N})N \bar{L}_{\e(N),R}
 \eta_r^N(x)dr +\tilde{m}_t^N-\tilde{m}_s^N,
\end{displaymath}
where $\tilde{m}_{\cdot}^N$ is a martingale which vanishes as $N$ goes to 0;
recall \eqref{eq:5.m^N}. On the other hand, the absolute value of the integral
term is bounded from above by
$\int^t_s 2\k\|\phi_{\k}\|_{2,\infty} dr$, 
recall \eqref{eq:5.7-a}. This concludes the proof of (iii) and therefore
the lemma.
\end{proof}

\subsubsection{Characterization of limit points}

We start with considering a class of martingales associated with 
$\{\tilde{\zeta}^N\}_{N\ge 1}$. Let $M_t^N(x), x\in \N$, be the martingale
defined by
\begin{displaymath}
M_t^N(x) := \zeta_t^N(x)  - \zeta_0^N(x)   - 
\int^t_0 N^2 \bar{L}_{\e(N),R} (\zeta_s^N(x))ds.
\end{displaymath}
Some simple computations permit us to rewrite 
\begin{align} \label{eq:martingale}
N^2 \bar{L}_{\e(N),R} (\zeta_s^N(x))
= N^2 \big( \e \zeta_s^N(x-1) -(\e+1) \zeta_s^N(x)
  + \zeta_s^N(x+1) \big),
\end{align}
for every $x\in \N$, where we define $\zeta_s^N(0)
:= \e^{-1} \zeta_s^N(2)$.  However, denoting $\b(N) := N(1-\e(N))$
which converges to $\b$ as $N\to\infty$ by Lemma \ref{lem:3.2},
the right hand side of \eqref{eq:martingale} can be rewritten further as
$$
N^2 \Delta \zeta_s^N(x) + N \b(N) \nabla \zeta_s^N(x),
$$
where $\De\zeta =(\Delta \zeta(x))_{x\in\N}$ and 
$\nabla\zeta=(\nabla \zeta(x))_{x\in\N}$ are defined 
for $\zeta= (\zeta(x))_{x\in \Z_+}$ by
\begin{align*}
\Delta \zeta(x) = \zeta(x-1)-2\zeta(x)+ \zeta(x+1), \quad
\nabla \zeta(x) = \zeta(x) - \zeta(x-1),
\end{align*}
respectively.  Thus, for every $g \in C^2_0(\R_+)$, 
taking account of $\zeta_s^N(0)= \e^{-1} \zeta_s^N(2)$, we have that
\begin{align}  \label{eq:5.5}
& \int^{\infty}_0 g(u) \tilde{\zeta}^N(t,u)du = \frac{1}{N} \sum_{x\in\N}
\zeta_t^N(x) g(x/N) + R_t^N  \\
& \qquad = \frac{1}{N} \sum_{x\in\N}\zeta_0^N(x) g(x/N) 
	+ \int^t_0 b^N(\zeta_s^N,g)ds + M_t^N(g) + R_t^N,
	\notag
\end{align}
where
\begin{align}  \label{eq:5.6}
b^N(\zeta,g) = & \frac{1}{N} \sum_{x\in\N} \Delta^N g(x/N)
 \zeta(x) - \frac{\b(N)}{N}\sum_{x\in\N} \nabla^N g(x/N)\zeta(x)  \\
 & \qquad + N\big(g(1/N)\zeta(2)-g(0)\zeta(1)\big),  \notag \\
\intertext{with}
\Delta^N g(x/N)& = N^2\big(g((x+1)/N) + g((x-1)/N) 
  -2 g(x/N) \big),    \notag \\
\nabla^N g(x/N) & = N\big(g((x+1)/N) - g(x/N) \big), 
  \quad  x\in \N,    \notag
\end{align}
and
\begin{displaymath}
M_t^N(g)=\frac{1}{N} \sum_{x\in\N}M_t^N(x) g(\frac{x}{N}).
\end{displaymath}
The error term $R_t^N$ in \eqref{eq:5.5} is defined by
$$
R_t^N = \int^{\infty}_0 g(u) \tilde{\zeta}^N(t,u)du 
- \frac{1}{N} \sum_{x\in\N}\zeta_t^N(x) g(x/N)
$$
and admits a bound:
\begin{equation} \label{eq:5.RN}
|R_t^N| \le e^{CX_t^N} \left\{\left(e^{C/N}-1\right) \|g\|_{L^1(\R_+)}
+ \frac1N \|g'\|_\infty\times |\text{supp }g|\right\}
\end{equation}
in view of \eqref{eq:5.Dzeta-1} and \eqref{eq:5.Dzeta-2}.  Therefore, 
Proposition \ref{prop:x} shows that $R_t^N$ tends to $0$ as $N\to\infty$
in probability.

The martingale term in \eqref{eq:5.5} vanishes in the limit:

\begin{lem}  \label{lem:5.6}
$E[M_t^N(g)^2]$ converges to $0$ as $N\to\infty$.
\end{lem}

\begin{proof}
A straightforward computation leads to the following results for the
quadratic and cross-variations of $M_t^N(x)$:
\begin{align*}
& \frac{d}{dt} \lan M^N(x)\ran_t
= \zeta_t^N(x)^2
\left\{ a_N c_-(x-1,\eta_t^N)
+ b_N c_+(x-1,\eta_t^N) \right\}, \quad x \ge 2, \\
& \frac{d}{dt} \lan M^N(1)\ran_t
= \zeta_t^N(1)^2
\left\{ a_N 1_{\{\eta_t^N(1)=0\}}
+ b_N 1_{\{\eta_t^N(1)=1\}} \right\}, \\
& \lan M^N(x), M^N(y)\ran_t =0, \quad 1 \le x \not= y,
\end{align*}
where $a_N=N^2(1-\e)^2/\e, b_N=N^2(1-\e)^2$.  This implies the conclusion
of the lemma.
\end{proof}

To treat the boundary term appearing in $b^N(\zeta,g)$ (i.e.\ the third term
in the right hand side of \eqref{eq:5.6}), we need the following
ergodic property of the $\eta$-process at the boundary site $\{1\}$.  Note
that this ergodic property holds at the single site $\{1\}$ without taking
any average over sites near the boundary as performed in \cite{ELS-2}.

\begin{lem}  \label{lem:ergod}
Under the condition \eqref{eq:5.prop.5.4} in Proposition \ref{prop:x},
for every $0 \le T_1 \le T_2 \le T$ and $\de>0$, we have that
\begin{displaymath}
\lim_{N \to \infty} \bar{\mathbb{Q}}_{\nu^N}^N
\left[\left|\frac{1}{T_2-T_1}\int^{T_2}_{T_1}\eta^N_s(1)
-\frac{1}{2}\right|>\de \right]=0.
\end{displaymath}
\end{lem}

\begin{proof}
Consider the martingale 
\begin{displaymath}
m_t^N:=X_t^N-X_0^N-\int^t_0 N^2 \bar{L}_{\e(N),R} (X_s^N)ds.
\end{displaymath}
By \eqref{eq:5.7-a} with $\fa\equiv 1$, we see that $N^2 \bar{L}_{\e(N),R}
(X_s^N) = N(1-2\eta^N_s(1)) -\b(N)(1-\eta^N_s(1))$.  However, since
Lemma \ref{lem:3.2} implies $0<\b(N)=N(1-\e(N))\le C$ for $N\ge 1$,
this proves that
\begin{displaymath}
\left|\int^{T_2}_{T_1} \{1-2\eta^N_s(1)\} ds\right|
\le \frac1N \left( X_{T_2}^N+ X_{T_1}^N 
+ |m_{T_2}^N| + |m_{T_1}^N| + CT_2 \right).
\end{displaymath}
Thus, the lemma follows from \eqref{eq:5.prop.5.4} and the estimate: 
$E[|m_T^N|^2]\le T$, which follows from \eqref{eq:5.m^N} with $\fa\equiv 1$.
\end{proof}

Once the following lemma for the boundary term in $b^N(\zeta,g)$
is established, the weak form \eqref{eq:5.w} of the equation \eqref{eq:omega}
is easily derived from \eqref{eq:5.5}, \eqref{eq:5.6},
\eqref{eq:5.RN} and Lemma \ref{lem:5.6}.
Thus, the proof of Theorem \ref{thm:hopf} is concluded by the uniqueness of
the weak solutions of \eqref{eq:omega}, which will be shown in the next
subsection.

\begin{lem}  \label{lem:5.8}
If $g \in C^2_0(\R_+)$ satisfies the condition
$2 g'(0) - \b g(0)=0$, then we have that
$$
\lim_{N\to\infty} \bar{\mathbb{Q}}_{\nu^N}^N\left[ \left| \int_0^T
N\left( g(1/N) \zeta_t^N(2) -g(0) \zeta_t^N(1) \right)dt \right|>\de
\right] =0
$$
for every $\de>0$.
\end{lem}

\begin{proof}
Recalling $\zeta_t^N(2) = \zeta_t^N(1) e^{(\log\e)\eta_t^N(1)}$,
we have that
$$
N\left( g(1/N) \zeta_t^N(2) -g(0) \zeta_t^N(1) \right)
= \zeta_t^N(1) \big( g'(0) -\b g(0) \eta_t^N(1) + r_t^N \big),
$$
where the error term $r_t^N$ is defined by
\begin{align*}
r_t^N := & \left\{N\big(g(1/N)-g(0)\big) - g'(0)\right\}
 + N \big(g(1/N)-g(0)\big) \left\{e^{(\log\e)\eta_t^N(1)} -1 \right\} \\
& \quad + Ng(0) \left\{ e^{(\log\e)\eta_t^N(1)} -1 + \b \eta_t^N(1)/N\right\},
\end{align*}
and tends to $0$ as $N\to\infty$ by Lemma \ref{lem:3.2}.
Therefore, by the boundary condition for $g$, if we can show that
\begin{equation} \label{eq:5.average}
\lim_{N\to\infty} \bar{\mathbb{Q}}_{\nu^N}^N\left[ \left| \int_0^T
\zeta_t^N(1) \big(\eta_t^N(1) -1/2\big) dt  \right|>\de\right] =0
\end{equation}
for every $\de>0$, the proof of the lemma is concluded.  However, as we have
shown in the tightness, the process $\{\zeta_\cdot^N(1)\}_{N\ge 1}$ has the
equi-continuity:
$$
\lim_{\ga\downarrow 0} \limsup_{N\to\infty} \bar{\mathbb{Q}}_{\nu^N}^N \left[
\sup_{\begin{subarray}{c} |t-s|\le \ga\\
0\le s < t\le T \end{subarray}}
  |\zeta_t^N(1)-\zeta_s^N(1)| > \de' \right] =0,
$$
for every $\de'>0$.  Therefore, if we divide the interval $[0,T]$
into small subintervals with length $\ga$:
$$
\left| \int_0^T \zeta_t^N(1) \big(\eta_t^N(1) -1/2\big) dt  \right|
\le \sum_{k=0}^{[T/\ga]}
\left| \int_{k\ga}^{(k+1)\ga\wedge T}
 \zeta_t^N(1) \big(\eta_t^N(1) -1/2\big) dt  \right|,
$$
$\zeta_t^N(1)$ in the integrand is close to $\zeta_{k\ga}^N(1)$ (if $\ga$ is
small enough) and we can apply Lemma \ref{lem:ergod} for the integral
$\int_{k\ga}^{(k+1)\ga\wedge T} \big(\eta_t^N(1) -1/2\big) dt$.  In other
words, $\zeta_t^N(1)$ changes slowly compared with the rapid motion of 
$\eta_t^N(1)$.  This proves \eqref{eq:5.average}.
\end{proof}

\begin{rem}  \label{rem:5.1}
For $g$ satisfying the same condition as Lemma \ref{lem:5.8},
a stronger assertion:
$$
\lim_{N\to\infty} \bar{\mathbb{Q}}_{\nu^N}^N\left[ \left| \int_0^T
N\sqrt{N} \left( g(1/N) \zeta_t^N(2) -g(0) \zeta_t^N(1) \right) dt 
\right|>\de \right] =0
$$
holds for every $\de>0$ even by multiplying an extra factor $\sqrt{N}$.
Indeed, this can be seen by noting that the error estimate given in
the proof of Lemma \ref{lem:ergod} is $O(1/N)$ and that in Lemma \ref{lem:3.2}
is $O(\log N/N^2)$ as $N\to\infty$.  This fact will be used in \cite{FS-1}.
\end{rem}

\subsubsection{Uniqueness of weak solutions}

Here, we prove the uniqueness of the weak solutions of \eqref{eq:omega}.
The method is standard, especially because the equation is linear.  We first
extend the class of test functions $g=g(u)$ in the weak form \eqref{eq:5.w}
to the family of all $g =g(t,u) \in C_0^{1,2}([0,T]\times\R_+^\circ)$ 
satisfying
$2 \partial_u g (t,0) - \b g(t,0)=0$ for every $t\in[0,T]$, and show that
\begin{align}  \label{eq:5.g}
\int^{\infty}_0 g(t,u) \omega(t,u)du =& \int^{\infty}_0 g(0,u)\omega(0,u)du\\
&+ \int^{t}_0 \int^{\infty}_0 \big(\partial_s g(s,u) + \partial_u^2 g(s,u)
-\b\partial_u g(s,u)\big) \omega(s,u)du ds   \notag
\end{align}
holds for every such $g$ and $t\in [0,T]$.  Indeed, this can be done by
dividing the interval $[0,t]$ into small pieces, assuming $g$ to be constant
in $s$ on each small interval, applying the weak form \eqref{eq:5.w}
on each such small interval and finally by passing to the limit.

Secondly, since the solution $\omega$ is assumed to be bounded, we can
further extend the class of $g$'s from functions having compact supports
in $[0,T]\times\R_+$ to those having the exponentially decaying property
as $u\to \infty$ in the sense that $\sup_{t\in[0,T],u\in\R_+} \{|g(t,u)| +
|\partial_t g(t,u)| + |\partial_u g(t,u)| + |\partial_u^2 g(t,u)|\} e^{ru}
<\infty$
for some $r>0$.  Finally, let $\fa\in C_0^\infty(\R_+)$ be given arbitrarily
and define $g\equiv g_\fa =g(t,u)$ as the solution of the backward equation:
\begin{equation*}
\left\{
\begin{aligned}
\partial_t g+ \partial_u^2 g- \b \partial_u g =0,& \quad t \in [0,T), 
\; u\in \R_+,  \\
g(T,u) = \fa(u),&  \quad u\in \R_+,  \\
2 \partial_u g(t,0) - \b g(t,0) =0,& \quad t \in [0,T).
\end{aligned}
\right.
\end{equation*}
Such $g$ exists and has the exponentially decaying property.  By choosing
this $g$ in \eqref{eq:5.g} with $t=T$, we obtain that
$$
\int^{\infty}_0 \fa(u) \omega(T,u)du
= \int^{\infty}_0 g_\fa(0,u) \omega(0,u)du,
$$
and this concludes the proof of the uniqueness of the weak solutions
of \eqref{eq:omega}.

\subsection{Proof of Theorem \ref{thm:fthm}}

We will show that Theorem \ref{thm:fthm} for the process $q_t$ follows from
Theorem \ref{thm:hopf} for the process $\eta_t$. To this end, 
we first see that the condition \eqref{eq:5.pi} is reduced from the 
condition \eqref{eq:2.fini} if we define $\eta$ and $\rho_0$ by 
$\eta=\eta^q$ and $\rho_0=-\psi_0^{\prime}$, respectively.  

For $g\in C_b^1(\R_+)$ satisfying $g(v)=0$ for $v\le 1/K$ and $g(v)=c$ for
$v\ge K$ with some $K>1$ and $c\in\R$, taking $g^{\prime}$ as $f$ in 
\eqref{eq:2.fini}, we have that
\begin{equation*}
\lim_{N \to \infty} \nu^N[|\int^{\infty}_0 g^{\prime}(u) \tilde{\psi}^N_q(u)
du - \int^{\infty}_0 g^{\prime}(u)\psi_0(u)du|>\de]=0
\end{equation*}
for every $\de>0$. By the definition, 
\begin{displaymath}
\int^{\infty}_0 g^{\prime}(u) \tilde{\psi}^N_q(u)du 
= \frac1N \sum_{i \in \Z_+}\int^{\frac{q_i}{N}}_0 g^{\prime}(u)du
= \frac1N \sum_{i \in \Z_+}g(\frac{q_i}{N})
=\frac{1}{N}\sum_{x \in \Z_+}g(\frac{x}{N})\eta^q(x).
\end{displaymath}
On the other hand, by the integration by parts formula,
\begin{displaymath}
\int^{\infty}_0 g^{\prime}(u)\psi_0(u)du=-\int^{\infty}_0 g(v)\psi^{\prime}_0(v)dv
=\int^{\infty}_0 g(v)\rho_0(v)du.
\end{displaymath}
Therefore, \eqref{eq:5.pi} is shown for functions $g$ satisfying the above
conditions.  However, this can be extended to a wider class of 
functions $g\in C_b(\R_+)$ satisfying $g(v)=c$ for $v\ge K$ with some $K>1$ and $c\in\R$, by approximating such $g$ by a sequence of continuous functions $g_n \in C_b^1(\R_+)$ satisfying $g_n(v)=0$ for $v\le 1/K$ and $g(v)=c$ for $v\ge K$ with some $K>1$ and $c\in\R$ noting that $0 \le \eta(x),\rho_0(v) \le 1$.

In order to complete the proof of the theorem, it is now sufficient to show that \eqref{eq:5.zeta} in Theorem \ref{thm:hopf} implies \eqref{eq:2.ft} with $\psi(t,u)=\frac{1}{\b}\log{\omega(t,u)}$. The non-linear equation \eqref{eq:fhydro} for $\psi_t$ follows from \eqref{eq:omega} for $\omega_t$. Especially, the boundary condition $2 \partial_u \omega(t,0) + \b \omega(t,0)=0$ implies that $\partial_u\psi(t,0)=-1/2$ and $\omega(t,\infty)=1$ implies that $\psi(t,\infty)=0$ for $t>0$.

Since $\tilde{\psi}_{q_t}^N(u)=\frac{1}{\b}\log \tilde{\zeta}^N(t,u)+o(1)$ 
with an error going to $0$ in probability as $N\to\infty$ in view of 
\eqref{eq:5.Dzeta-1} and \eqref{eq:5.Dzeta-2}, noting that $\omega(t,u),
\tilde{\zeta}^N(t,u) \ge 1$, \eqref{eq:5.zeta} implies that
\begin{equation*}
\lim_{N \to \infty} \bar{\mathbb{Q}}_{\nu^N}^N
\left[\sup_{0 \le t \le T, 0 \le u \le K}
|\tilde{\psi}_{q_t}^N(u)-\psi(t,u)|>\de\right]=0.
\end{equation*}
for every $T>0$, $K>0$ and $\de>0$. This completes the proof of 
Theorem \ref{thm:fthm}.

\vskip 7mm
\noindent
{\bf Acknowledgement} $\;$ 
The authors thank H.\ Spohn for leading them to the problems discussed
in this paper.


\begin{thebibliography}{99}
\bibitem{AD}{\sc D.\ Aldous and P.\ Diaconis},
{\it Longest increasing subsequences: from patience sorting to the 
Baik-Deift-Johansson theorem},
Bull.\ Amer.\ Math.\ Soc., {\bf 36} (1999), 413--432. 

\bibitem{CS}{\sc L.\ Chayes and G.\ Swindle}, 
{\it Hydrodynamic limits for one-dimensional particle systems with
moving boundaries}, Ann.\ Probab., {\bf 24} (1996), 559--598.

\bibitem{CSS}{\sc L.\ Chayes, R.H.\ Schonmann and G.\ Swindle}, 
{\it Lifshitz's law for the volume of a two-dimensional droplet 
at zero temperature}, J.\ Statist.\ Phys., {\bf 79} (1995), 821--831.  

\bibitem{CK}{\sc L.\ Chayes and I.C.\ Kim}, 
{\it A Two-Sided Contracting Stefan Problem},
Comm.\ Partial Differential Equations, {\bf 33} (2008), 2225--2256. 

\bibitem{DVZ}{\sc A.\ Dembo, A.\ Vershik and O.\ Zeitouni},
{\it Large deviations for integer partitions},
Markov Processes Relat.\ Fields, {\bf 6} (2000), 147--179.

\bibitem{ELS-2}{\sc G.\ Eyink, J.\ L.\ Lebowitz and H.\ Spohn}
{\it Lattice gas models in contact with stochastic reservoirs: Local 
equilibrium and relaxation to the steady state},
Commun.\ Math.\ Phys, {\bf 140} (1991), 119--131.

\bibitem{Fu}{\sc T.\ Funaki},
{\it Stochastic Interface Models}, 
Ecole d'Et\'e de Probabilit\'es de Saint-Flour XXXIII -- 2003, 
pp.\ 103--274, Lect.\ Notes Math., {\bf 1869}, Springer, Berlin, 2005.

\bibitem{FS-1}{\sc T.\ Funaki and M.\ Sasada},
{\it Fluctuations in an evolutional model of two-dimensional Young diagrams},
in preparation.

%\bibitem{FS-2}{\sc T.\ Funaki and M.\ Sasada},
%{\it Hydrodynamic limit for an evolutional model of two-dimensional
%Young diagrams associated with generalized uniform statistics},
%in preparation.

\bibitem{G}{\sc J.\ G\"{a}rtner},
{\it Convergence towards Burger's equation and propagation of chaos for 
weakly asymmetric exclusion processes}, 
Stochastic Process.\ Appl., {\bf 27} (1988), 233--260.

\bibitem{J1}{\sc K.\ Johansson},
{\it Shape fluctuations and random matrices},
Commun.\ Math.\ Phys., {\bf 209} (2000), pp.\ 437--476.

\bibitem{J2}{\sc K.\ Johansson},
{\it Random growth and random matrices},
European Congress of Mathematics, Vol. I (Barcelona, 2000),
pp.\ 445--456, Progr.\ Math., {\bf 201}, Birkh\"auser, Basel, 2001.

%\bibitem{K}{\sc S.V.\ Kerov}, 
%{\it Asymptotic Representation Theory of the Symmetric Group 
%and its Applications in Analysis},
%Translations of Math.\ Monographs, {\bf 219}, 
%Amer.\ Math.\ Soc., Providence, RI, 2003, xvi+201 pp.

%\bibitem{KL}{\sc C.\ Kipnis and C.\ Landim},
%{\it Scaling Limits of Interacting Particle Systems},
%Springer, 1999.

\bibitem{KOV}{\sc C.\ Kipnis, S.\ Olla and S.R S.\ Varadhan},  
{\it Hydrodynamics and large deviation for simple exclusion processes},
Comm.\ Pure Appl.\ Math., {\bf 42} (1989), 115--137.

\bibitem{LM}{\sc C.\ Landim and M.\ Mourragui},
{\it Hydrodynamic limit of mean zero asymmetric zero range processes in 
infinite volume},
Ann.\ Inst.\ Henri Poincar\'e (B), {\bf 33} (1997), 65--82.

\bibitem{P}{\sc B.\ Pittel},
{\it On a likely shape of the random Ferrers diagram},
Adv.\ Appl.\ Math., {\bf 18} (1997), 432--488.

\bibitem{Sh}{\sc S.\ Shlosman},
{\it The Wulff construction in statistical mechanics and combinatorics},
Russian Math.\ Surveys, {\bf 56} (2001), pp.\ 709--738.

\bibitem{Sp}{\sc H.\ Spohn},
{\it Interface motion in models with stochastic dynamics},
J.\ Statis.\ Phys., {\bf 71} (1993), pp.\ 1081--1132.

\bibitem{V}{\sc A.\ Vershik},
{\it Statistical mechanics of combinatorial partitions and their limit shapes}, Func.\ Anal.\ Appl., {\bf 30} (1996), 90--105.
\end{thebibliography}
\end{document}